\documentclass{amsart}  
\usepackage[parfill]{parskip}
\usepackage{graphicx, verbatim}
\usepackage{appendix}
\usepackage{amsfonts}
\usepackage{url}
\usepackage{hyperref} 
\hypersetup{backref,pdfpagemode=FullScreen,colorlinks=true}
\usepackage{amsmath}
\usepackage{amssymb}
\usepackage{tikz, tikz-cd}
\usetikzlibrary{arrows.meta}
\usepackage{amscd}
\usepackage{color}
\usepackage{amsthm}
\usepackage{bm}
\usepackage{indentfirst}
\usepackage[hmargin=3cm,vmargin=3cm]{geometry}
\numberwithin{equation}{section}

\newtheorem{theorem}[equation]{Theorem} 
\newtheorem{proposition}[equation]{Proposition}
 
\newtheorem{lemma}[equation]{Lemma} 
 
\newtheorem{corollary}[equation]{Corollary} 
 
\newtheorem{conjecture}{Conjecture}
\theoremstyle{definition}
\newtheorem{definition}[equation]{Definition}

\newtheorem{terminology}{Terminology}

\theoremstyle{remark}

\newtheorem{remark}[equation]{Remark}
\newtheorem{example}{Example}
\newtheorem{question}{Question}

\DeclareMathOperator {\mult} {mult}

\DeclareMathOperator{\image}{\mathrm{image}}

\begin{document}
% \href{http://yashamon.github.io/web2/papers/conformalsymplectic.pdf}{Direct link to author's version}
\title{Geodesic string counting invariants and arithmetic of 
multiplicities}
\author{Yasha Savelyev}
\thanks {Supported by CONAHCYT 
research grant CF-2023-I}
\email{yasha.savelyev@gmail.com}
\address{Faculty of Science, University of Colima, Mexico}
\begin{abstract} 
We study rational valued counts of geodesic strings (reparametrization equivalence classes of closed geodesics) for complete Riemann-Finsler manifolds,
based on the Fuller index of the geodesic flow. The main conceptual
result is a product type formula for these counts. Combined with
aspects of KAM theory, it yields the following sample phenomenon. Let
$g$ be a generic Finsler metric on $T ^{2}$, sufficiently $C ^{\infty
}$-close to a flat metric, and fix a prime $p$ and a nontrivial free
homotopy class $\beta $. If there is a class $\beta$ 
$g$-geodesic string with multiplicity divisible by $p$, 
then there is another one. We
also obtain arithmetic constraints on counts of closed geodesics in
mapping tori and flat bundles, and constraints on the existence of
negative sectional curvature metrics. These counts can be understood
as a shadow of a conjectural orbifold Morse homology of the
infinite-dimensional quotient stack $[LX/S^1]$.
\end{abstract}
\maketitle
% \tableofcontents 
\section{Introduction}
A \textbf{\emph{geodesic string}} will be a short name for 
the equivalence class of a closed, constant speed
$g$-geodesic up to reparametrization $S ^{1}$ action. 
Suppose for the moment $X$ is compact. For a nontrivial free homotopy class $\beta$ of loops in
$X$, we say that a metric $g$ on $X$ is $\beta $-\textbf{\emph{regular}} if all
of its class $\beta $ closed geodesics are nondegenerate
in the usual sense, or equivalently the closed orbits of the
associated geodesic flow are dynamically nondegenerate. 
For a $\beta $-regular $g$, with finitely many class $\beta
$ geodesic strings we define:
\begin{equation} \label{eq_Fgb}
F (g, \beta) = \displaystyle \sum_{o} \frac{(-1) ^{\operatorname {morse}  (o)}}
{\mult (o)} \in \mathbb{Q}, 
\end{equation}
where: 
\begin{itemize}
	\item The sum is over class $\beta$  $g$-geodesic strings
	$o$. 
	\item $\operatorname {morse} (o)$ 
denotes the Morse index of $o$, (meaning the Morse-Bott index
of the associated critical submanifold (diffeomorphic to $S
^{1}$) of the loop space).
\item $\mult (o) $ is the geometric multiplicity, equivalently
the order of the corresponding isotropy subgroup of $S
^{1}$,  where $S ^{1}$ is acting on the loop space by reparametrization.
\end{itemize}

We study invariance properties and computations of the above
$F$-count, and it's generalizations using the Fuller index
~\cite{cite_FullerIndex}. This also leads to some novel 
applications, particularly to arithmetic properties of 
geodesic multiplicities. One of the principal tools is a 
certain product type formula  for  the $F$-count. While the 
Fuller index itself is a very general construction in 
dynamical systems,  the product formula is a more subtle
phenomenon which needs the Riemann-Finsler geometric context.   
A detailed outline of Fuller's theory and basic
definitions are given here in the Appendix, so no prior
knowledge is necessary.

For a free homotopy class $\beta$ of loops in $X$, let
$L_{\beta}X$ denote the corresponding component of the free loop space. 
The following is a direct corollary of Theorem
\ref{theorem_Eulercharacteristic}, and exemplifies the
basic topological invariance of the $F$-count:
\begin{corollary} \label{cor_chi} Let $X,g, \beta $ be as above 
and such that $\beta$ is indivisible, (cannot be
represented by a multiply covered loop).
Then the $S ^{1}$-equivariant homology $H ^{S ^{1}} _{*} (L
_{\beta}X, \mathbb{Z})$ has finite total rank. Denote by $\chi ^{S ^{1}} (L _{\beta} X)$ the Euler characteristic of this homology.
Then 
\begin{equation*}
F (g, \beta ) =  \displaystyle \sum_{o} {(-1)
^{\operatorname {morse}  (o)}} = \chi ^{S ^{1}}(L _{\beta
}X),
\end{equation*}
(note $\mult (o) =1$ by the indivisible condition.) 
\end{corollary}
The proof of the above is based on $S ^{1}$-equivariant
Morse theory.
For divisible classes, the homology theoretic interpretation of the $F$-count would require more than ordinary
$S^1$-equivariant Morse homology for the $g$-energy functional.
One would want a $g$-invariant up to quasi-isomorphism ``orbifold Morse
homology chain complex'' for the quotient stack
$[LX/S^1]$, together with (assuming certain finiteness) an algebraically derived
$\mathbb{Q} $-valued characteristic\footnote {I am told that
in this context a better term is von Neumann dimension.} coinciding with the
$F$-count. This is likely related to Wall-Euler
characteristic theory in group cohomology theory, which also
appears in the more basic Conjecture \ref{con_aspherical}.
Existence of this orbifold Morse homology is at the moment
unknown; for a finite dimensional candidate
see~\cite{cite_ChoHong2014OrbifoldMorse}.  

For a compact $X$, we say a metric $g$ is taut if the space of class $\beta $ closed
$g$-geodesics is compact for every $\beta $. (With the
topology induced by the $C ^{1}$ topology on the loop
space.) 
It is said to be $\beta $-taut if the same holds for
just $\beta $-class geodesics. 
A homotopy $\{g _{t}\} _{t
\in [0,1]}$  is
said to be $\beta $-taut if the space of pairs $(o,t)$  with
$o$  class $\beta $ $g _{t}$-geodesic is compact.
An example of a metric $g$, $\beta $-taut homotopic to the
flat metric is in Figure
\ref{fig:torus-beta-two-example}. We can construct such examples
which are arbitrarily ``far'' from being flat. Figure
\ref{fig:torus-beta-two-example} is also an example for the
following.

\begin{figure}[t]
    \centering
    \includegraphics[width=.9\textwidth]{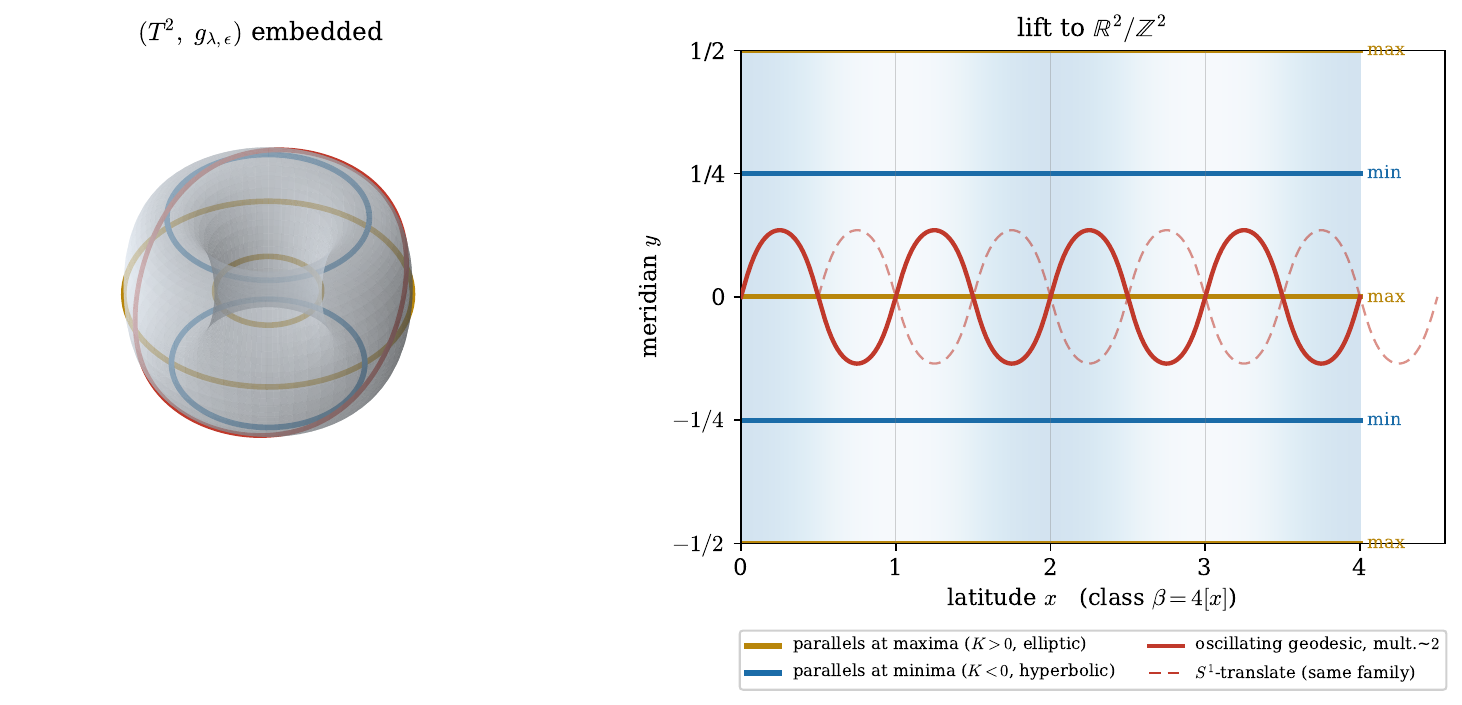}
    \caption{
    Schematic of a $\beta $-taut torus metric with mixed
		geodesic multiplicities in the class
    $\beta=4[x]$.  In coordinates
    \[
        T^2=\mathbb R^2/\mathbb Z^2,\qquad
        g_{\lambda,\epsilon}
        =
        dy^2+\lambda^2 f_\epsilon(y)^2\,dx^2,
        \qquad
        f_\epsilon(y)=1+\epsilon\cos(4\pi y),
    \]
    the $x$-direction is drawn as the latitude direction.  The four critical
    levels of $f_\epsilon$,
    \[
        y=0,\quad \frac14,\quad \frac12,\quad \frac34,
    \]
    give four geodesics; their 4-covers are class $\beta$
    geodesic strings of multiplicity $4$.  For
    $\lambda > (4 \sqrt {\epsilon(1+ \epsilon)}) ^{-1}$
		there are, in addition, two $S ^{1}$-families ($S ^{1}$
		here is acting by rotation in the $x$-direction) of
		non-horizontal oscillating geodesic strings in
    the same class $\beta$, one about each maximum of $f
		_{\epsilon }$, each of multiplicity $2$. This is
		in part classical Clairaut theory, see for instance
		~\cite{cite_AlexanderClosedGeodesicsRevolution},
		\cite[§4-4, Example 5]{cite_doCarmoCurvesSurfaces}. 
    After a small generic
    perturbation each family breaks into isolated strings
		cancelling in the $F$-count. There are also
		multiplicity-one oscillating $\beta $-class geodesics
		(one oscillation while winding four times), already for
		$\lambda > (8 \sqrt {\epsilon(1+ \epsilon)}) ^{-1}$;
		further multiplicity-one families appear as $\lambda $
		increases. In particular, we can
		see how Theorem \ref{thm_IntroTnH} is satisfied in this
		example taking $p=2$. To see the $\beta $-taut homotopy
		to the flat metric take $$g_s = dy^2
		+ \lambda^2\big(1+s\epsilon\cos(4\pi
		y)\big)^2 dx^2,\qquad s:1\to 0.$$ Each $g _{s}$ is
		Clairaut-integrable and tautness is then
		established by explicit tracking of geodesics. 
    }
    \label{fig:torus-beta-two-example}
\end{figure}
% \begin{figure}[h]
%  \includegraphics[width=2.0in]{torus.pdf}
% % \scalebox{.9}{\input{Tree.pdf}}
%  \caption {The class $\beta $ is represented by the drawn
%  curve.} \label{figure_torus}
% \end{figure}   

% Even this local form appears to be
% inaccessible to standard Morse theory techniques, 
% is a new phenomenon and is likely the most central application
% of the paper.   

% The following theorem in particular says that for regular
% metrics on $T ^{2}$ sufficiently nearby to the flat metric, fixed class geodesic strings with ``small length'' must come at
% least in pairs, corresponding to multiplicity constraints.
\begin{theorem} [Proof in Section \ref{sec_Proof of Theorem thm_IntroTnH}]\label{thm_IntroTnH}
Let $g$ be a $\beta $-regular  Riemann-Finsler metric on $T ^{n}$, $n>1$, $\beta $-taut homotopic to the flat metric for $\beta \in \pi _{1} (T ^{n})$ a nontrivial class.  % Suppose that $\beta \in \pi _{1} ^{inc} (X)$ is in the image
% of the inclusion $\pi _{1} ^{inc} (Z) \to \pi _{1} ^{inc} (Z
% \times Y)$.
Fix a prime $p$. Suppose there is a string $o$ such that:
\begin{enumerate}
\item $o$ is  $g$-geodesic. 
\item $o$ has class $\beta $.
\item $p | \mult (o) $.
\end {enumerate}
Then there is another string satisfying these conditions. 
\end{theorem}
To prove the above we use the product
formula of Theorem \ref{thm:EulerProduct} to show that for
such $g$, $F (g, \beta )=0$. The conclusion then follows by
basic arithmetic.  

Using geodesic confinement based on KAM theory we show in Section \ref{sec_nearflat} the following:
\begin{theorem} \label{exm_tautT2} Any Finsler metric $g$ on $T ^{2}$ sufficiently $C
^{k}$-close to a given flat metric, for a universal $k$
(see Proposition \ref{prop_nearflat}), is
taut homotopic to the flat metric.
\end{theorem}
A Finsler metric that is $\beta $-regular for all $\beta $ is $C
^{\infty}$-generic by the bumpy metric theorem, in the Riemannian case
due to Abraham \cite{cite_Abraham1970Bumpy} and Anosov
\cite{cite_Anosov1982Bumpy} and in the Finsler case to
Rademacher--Taimanov \cite{cite_RademacherTaimanov2020Finsler}; hence
the above yields the following. 
\begin{corollary} 
Let $g$ be a generic Finsler metric on $T ^{2}$
sufficiently $C ^{k}$-close to a flat metric, $k$ as in Theorem \ref{exm_tautT2}.  
% Suppose that $\beta \in \pi _{1} ^{inc} (X)$ is in the image
% of the inclusion $\pi _{1} ^{inc} (Z) \to \pi _{1} ^{inc} (Z
% \times Y)$.
Fix a prime $p$ and a nontrivial $\beta $. Suppose there is
a string $o$ such that:
\begin{enumerate}
\item $o$ is  $g$-geodesic. 
\item $o$ has class $\beta $.
\item $p | \mult (o) $.
\end {enumerate} 
Then there is another string satisfying these conditions. 
\end{corollary}
Theorem \ref{exm_tautT2} settles the following question near
the flat metric; the global case is open, so at present the
hypothesis of Theorem \ref{thm_IntroTnH}, of being $\beta $-taut
homotopic to the flat metric, is verified only near the flat metric.
\begin{question} \label{que_T2global}
Is every $\beta $-taut metric on $T ^{2}$ $\beta $-taut
homotopic to the flat metric?
\end{question}

% In Section \ref{sec_Proof of Theorem thm_IntroTnH} we will
% give a more qualified form of this result which applies to
% metrics on $T ^{n}$ for all $n>1$. 
% \begin{question} \label{que_}
% Do $\beta $-taut
% metrics on $T ^{n}$ not $\beta $-taut homotopic to the flat
% metric exist?
% \end{question}
The above is not accessible to $S
^{1}$-equivariant Morse homology techniques, as the latter
doesn't suitably track isotropy information.  Nevertheless, there is a family resemblance with the resonance
identities of closed geodesic and closed characteristic theory, where
finite collections of periodic objects satisfy identities involving
average Euler characteristics and mean indices; see, for example,
~\cite{cite_Rademacher1989AverageIndices,
cite_HingstonRademacher2013Resonance} and the later noncontractible
resonance identities ~\cite{cite_LiuXiao2017ResonanceRPn}. The
constraints of the present paper are of a different character: the
resonance identities are real valued relations driven by the
asymptotics of the index, while ours are divisibility relations among
the multiplicities, driven by the vanishing of a rational count; in
particular they constrain the multiplicities modulo a prime, rather
than through indices or lengths. 

The near flat statements above extend to Reeb flows on the unit
cotangent bundle; the following is proved in Section
\ref{sec_ReebNearFlat}.
\begin{corollary}
\label{cor_reebintro}
There is a universal $k$ with the following property. Let $g_0$ be a
flat metric on $T^2$, let $C = S^*_{g_0}T^2$ with its canonical
contact form $\lambda_{g_0}$, and let $R$ be the Reeb vector field of
a contact form $f\lambda_{g_0}$ on $C$ with $\|f - 1\|_{C^k}$
sufficiently small. Then, in the flat trivialization $C
\cong T^2 \times S ^{1}$,
every closed $R$-orbit has the free homotopy class $(\beta, 0)$ of a
lifted flat geodesic, for some $\beta \neq 0$; in particular $R$ has
no contractible closed orbits. Moreover $R$ is taut homotopic,
through Reeb vector fields, to the flat geodesic Reeb field
$R^{\lambda_{g_0}}$; and if all class $(\beta,0)$ closed $R$-orbits
are nondegenerate and one class $(\beta, 0)$ orbit string $o$ satisfies $p \mid \mult(o)$
for a prime $p$, then there is another such orbit string.
\end{corollary}

% Fuller's insight is that such a count  (in the more general
% context of orbits of dynamical systems) 
% These invariants can be
% directly interpreted as the untwisted part of certain elliptic Gromov-Witten invariants in an associated lcs manifold,
% ~\cite{cite_SavelyevEllipticCurvesLcs}. The twist is coming by
% way of metric isometries. 
% The latter count in modern
% terms is just an orbifold count of points of virtual
% dimension 0 Kuranishi spaces, corresponding to spaces of
% geodesic strings in a fixed free homotopy class. But
% Fuller's construction makes the latter elementary and
% geometric.
% Corollary \ref{cor_chi} above, and its parent Theorem
% \ref{theorem_Eulercharacteristic} lead us to conjecture full
% topological invariance of the $F$-count, see Conjecture \ref{con_topological} in
% Section \ref{sec_Definition of the invariant F}. The latter theorem is an important ingredient for some geometric applications concerning closed geodesic counts.

We will show that the count $F
(g,\beta ) \in \mathbb{Q} $ can be defined for any $\beta
$-taut $g$ (Section \ref{sec:Definition of F}). What
follows is a special version of the Conjecture
\ref{con_topological} (appearing further ahead) and can be understood as a divisible
analogue of Corollary \ref{cor_chi} in the aspherical case.
\begin{conjecture}[Aspherical topological formula for $F$]
\label{con_aspherical}
Let $X$ be a closed aspherical manifold, and let $\beta$  be
a nontrivial free homotopy class. Suppose that $X$ admits a $\beta$-taut metric.
Then the Borel construction
\[
        (L_\beta X)_{hS^1}:=L_\beta X\times_{S^1}ES^1
\]
is Wall-Euler finite, and for every $\beta$-taut metric $g$,
\[
        F(g,\beta)
        =
        \chi_{\operatorname{Wall}}\bigl((L_\beta X)_{hS^1}\bigr).
\]
Equivalently, set $G=\pi_1 (X)$ let $C_G(\gamma_\beta)$ be the centralizer of $\gamma _{\beta }$ in $G$ and
$Q_\beta=C_G(\gamma_\beta)/\langle\gamma_\beta\rangle$. 
If $\beta=[\gamma_\beta]$
(the conjugacy class of $\gamma _{\beta } \in G$), and if
then $Q_\beta$ is Wall-Euler finite and
\[
        F(g,\beta)=\chi_{\operatorname{Wall}}(Q_\beta).
\]
\end{conjecture} The conjecture implies that $F(g,\beta)$ is independent of the
$\beta$-taut metric $g$, for closed aspherical manifolds.
For background on rational Euler characteristics of groups, we refer to
Wall's original note~\cite{cite_Wall1961} and to Brown's treatment of Euler
characteristics of groups in~\cite[Chapter IX]{cite_Brown1982}; see also
Brown~\cite{cite_Brown1974} for Euler characteristics of discrete groups and
$G$-spaces.

Note that if we assume the above conjecture, then Theorem \ref{thm_IntroTnH}
holds for any $\beta $-regular $g$ with finitely many class $\beta $ $g$-geodesic
strings.

% Assuming the conjecture we get a generalization of
% Theorem \ref{thm_IntroTn}.
% \begin{theorem} [Proof in Section
% \ref{sec_proofofThmthm_introConjecture}]\label{thm_introConjecture}
% Assume the above conjecture, let $\beta \in
% \pi _{1} (T ^{n}) $,  $n>1$, be a nontrivial class  and let
% $g$ be a $\beta $-regular
% Finsler metric on $T ^{n}$ having finitely many $\beta
% $ class geodesic strings. 
% Suppose that $o$ is a $g$-geodesic
% string s.t.:
% \begin{enumerate}
% 	\item $o$ has class $\beta $.
% 	\item $p | \mult (o) $, where $p$ is prime,
% \end{enumerate}
% then there is at least one other such $o$ (in particular same $p$).
% \end{theorem}
% \begin{remark} \label{rem_reeb_conjecture}
% We can extend the above from metrics to Reeb vector fields  
% on the unit cotangent bundle, representing the standard contact
% structure. But now counting closed Reeb orbits in classes
% $\widetilde{\beta }$, lifting a class $\beta $ as above,
% (fix a constant speed parametrized representative and lift). 
% However, we need 
% a stronger form of the Conjecture \ref{con_topological},
% see Remark \ref{rem_Reeb}. 
% \end{remark}

% Recall that a now classical theorem of Preissmann
% ~\cite{cite_PreissmanNegativeCurvatureProducts} implies that
% there are no nontrivial compact products with a negative
% sectional curvature Riemannian metric. The following is one
% generalization. 
The following is a rather important open question in our
context:
\begin{question} \label{que_tautaspherical} Suppose that $X$ is closed,
does $\beta$-taut for some $\beta $ imply aspherical? 
\end{question}
\begin{proposition} \label{prop_infdivisible}
Let $X$ be a closed manifold and let $\beta \neq 1$ be a free
homotopy class which is infinitely divisible in $\pi _{1} (X)$: for
arbitrarily large $n \in \mathbb{N}$, a representative of $\beta $ in
$\pi _{1} (X)$ is an $n$-th power. Then $X$ admits no $\beta $-taut
metric.
\end{proposition}
\begin{proof}
Take any metric $g$ and let $l>0$ be the systole of $g$, that is
$$l = \inf _{\{o \,|\, \text{o is noncontractible} \}} \ell
_{g} (o).$$ For $n$ as in the statement, write a representative of
$\beta $ as $h _{n} ^{n}$, with $h _{n}$ necessarily noncontractible
since $\beta \neq 1$. A length-minimizing closed geodesic in the free
homotopy class of $h _{n}$ exists as $X$ is closed, and its $n$-fold
cover is a class $\beta $ closed $g$-geodesic of length at least $n
\cdot l$. So the class $\beta $ closed geodesics of $g$ have unbounded
length and $g$ is not $\beta $-taut.
\end{proof}
\begin{remark} \label{rem_infdivisible}
A possible route to a counterexample for
Question~\ref{que_tautaspherical} is a mapping torus of a
non-aspherical $Z$ by a diffeomorphism with ``hyperbolic behavior''.
In the other direction, asphericity does not by itself help: by
results of Mess~\cite{cite_Mess1990} and
Belegradek~\cite[Corollary 5.1]{cite_Belegradek2006}, in every
dimension $n \geq 4$ there are closed aspherical $n$-manifolds whose
fundamental groups contain nontrivial infinitely divisible elements,
and hence, by Proposition \ref{prop_infdivisible}, classes $\beta $
for which no $\beta $-taut metric exists.
\end{remark}

\subsection{Geodesics in mapping tori and negative sectional
curvature} \label{sec_Curvature constrains}
A theorem of Preissmann ~\cite{cite_PreissmanNegativeCurvatureProducts} says that there are no negative sectional curvature metrics on compact products.
% for the $F$-count, whose basic corollary is the following. 
For fibrations, the most obvious analogue of Preissmann's theorem fails even
assuming compactness. 
By Thurston's fibered hyperbolization theorem, a closed surface bundle over
$S^1$ with fiber of negative Euler characteristic is
hyperbolic if and only if its monodromy is pseudo-Anosov; see
\cite{cite_ThurstonHyperbolicStructuresII,cite_ThurstonClassificationSurfaceDiffeo}.
% In fact,  by a classical result of Thurston
% every closed 3-manifold $X $, for which there is no injection
% $\mathbb{Z} ^{2} \to \pi _{1} (X, x _{0})$, and which
% fibers over a circle has a hyperbolic structure $g _{h}$.
What follows is a kind of partial, ``only if'' Finsler
analogue in higher dimensions, noting that an isometry of
a manifold has ``anti-Anosov behavior'' (for some
reasonable generalization of pseudo-Anosov to higher dimensions). 

We denote by $\pi _{1} ^{inc} (X)$ the set of free homotopy
classes of loops incompressible to the ends, in the sense of
Definition \ref{definition_boundaryincompressible},
(by convention nonconstant classes when $X$ is closed). 
\begin{corollary} [Corollary of Corollary
\ref{corollary_Thurston0}] \label{cor_isometry}
Let $X _{\phi }$ be the mapping torus of an isometry
$\phi: Z \to Z $ of a complete taut
Riemannian metric on $Z$, with $Z$ finite type, and where $\beta _{Z} \in \pi _{1} ^{inc} (Z)$ is indivisible. 
Let $\beta \in \pi _{1} ^{inc} (X _{\phi}) $ denote the
image of $\beta _{Z}$ under inclusion.   Then:
\begin{enumerate}
	\item $\chi ^{S ^{1}}(L _{\beta} X _{\phi } ) =0 $. 
	\item $X _{\phi}$ does not admit a forward complete,
$\beta $-regular Riemann-Finsler metric with an odd number
of class $\beta $ geodesic strings.
\item  $X _{\phi }$  does not admit a complete
negative sectional curvature Riemannian metric.
\end{enumerate}

\end{corollary}
\begin{example} \label{exm_surface} The corollary applies to
the mapping torus of a finite type surface of genus at least
one (the genus condition allows us to find a nontrivial class in $\pi _{1}
^{inc} (X _{\phi}) $, we also need a suitable taut
metric e.g. a nonpositive curvature metric, with an isometry). The conclusion that $X _{\phi}$ does not admit
a complete negative sectional curvature metric, follows from
Thurston's classification\footnote {Suitably adjusted, one
may also need finite geometry (volume) assumptions.}. However
the arithmetic constraints on geodesic counts seem
interesting even for a surface map.
% Technically, the more difficult
% component to the corollary is the arithmetic constraint on the number of
% $\beta $-class geodesic strings. Just the curvature
% constraint conclusion
% should be approachable by more standard ideas.
\end{example}

% \begin{proof} [Proof]
% This follows by Corollary \ref{corollary_Thurston0} once we
% note that the isometry group $\operatorname {Isom} $  of
% a closed Riemannian manifold is compact, hence $\pi _{0}
% \operatorname {Isom}   $ is finite.
% \end{proof}
% The first part of the theorem above can likely be proved directly via the 
% flat strip theorem ~\cite{cite_FlatStripEberleinOneil}, and
% other classical differential geometric tools. 
If we further assume that $Z$ is aspherical, the first part of Corollary \ref{cor_isometry} can
be proved via group cohomology theory, for which we already
provided some references.
But the main point is that in our setup the theorem has an elementary proof, once the
general product formula of
Theorem \ref{thm:EulerProduct} for the $F$-count is established. The proof of the latter 
does involve a tricky perturbation, but Fuller index theory
allows us to do this abstractly, perturbing a vector field
rather than the metric/energy functional.

% \begin{definition}\label{definition_Hadamardequivalent}
% We say that a pair $g _{0}, g _{1}$ of forward complete
% Finsler metrics with
% nonpositive flag curvature are
% \textbf{\emph{Hadamard equivalent}} if there is
% a continuous interpolation $\{g _{t}\}$, $t \in [0,1] $,
% (with respect to the topology of $C
% ^{0}$ convergence on compact sets) s.t. each $g _{t}$ has
% nonpositive flag curvature.
% \end{definition}

% \begin{theorem} \label{thm_}
% Let $\Sigma $ be as in Part 2 of Theorem
% \ref{thm_IntrojOOOOOoTn0}, or $\Sigma $ is as $Z$ in the
% theorem just above. 
% Suppose that $g$ is a Finsler metric on $M = \Sigma \times
% T ^{n}$ s.t. there are
% finitely many closed $\beta $ class geodesic strings in
% $M$ and they are all nondegenerate. Suppose
% further that $o$ is a $g$-geodesic string s.t.:
% \begin{enumerate}
% 	\item $o$ has class $\beta $.
% 	\item $p | \mult (o) $, where $p$ is prime.
% \end{enumerate}
% Then there is at least one other $g$-geodesic
% string satisfying these conditions.
%
% \end{theorem}

% (For example, $X,g _{0}$ in what follows can be
% any $\beta $-taut fibration as defined in Section
% \ref{example_indivisible}.) 
% My understanding is that reproducing such results with $S
% ^{1}$-equivariant homology computations will be extremely
% difficult.
% The condition $g$ being $\beta$-taut homotopic to $g _{st}$
% can be rephrased as 
\subsection*{Organization of the paper}
Section \ref{sec_Statement of results} sets up geodesic strings, taut
metrics and taut deformations, states the supporting theorems, and
contains the applications to sky catastrophes. Section
\ref{sec_Compactness preliminaries} proves the compactness results
underlying tautness, in particular Theorem \ref{thm:notgeodesible}.
Section \ref{sec_Some preliminaries on Reeb dynamics} collects
Reeb-flow preliminaries, and Section \ref{sec:Definition of F}
constructs the invariant $F$ and proves its taut deformation
invariance. Section \ref{sec_Proof of Theorem
theorem_Eulercharacteristic} proves the Euler characteristic formula
for indivisible classes, and Sections \ref{sec_Product Formula} and
\ref{sec:Proof of Theorem fibration} prove the product and fibration
formulas. Section \ref{sec_nearflat} proves the near flat rigidity
Proposition \ref{prop_nearflat} by KAM confinement, Section
\ref{sec_Proof of Theorem thm_IntroTnH} deduces the torus theorems of
the introduction, and Section \ref{sec_ReebNearFlat} extends the
near flat results to Reeb flows. Appendix \ref{appendix:Fuller}
develops the Fuller index and its continuation theory in the
noncompact setting, Appendix \ref{sec_A local fibered calculation}
performs a local fibered computation of $F$ used in the product
formula, and Appendix \ref{appendix_bluesky} contains the facts about
sky catastrophes used in Section \ref{sec_digression}.

\section{Setup and supporting theorems} \label{sec_Statement of results}
Proofs of many results here are postponed till subsequent
sections.
% We denote by $0
% \in \pi _{1} (X)$, the class of the constant loop. (For any connected component.)
% First helpful technical definitions. \begin{definition}\label{def_complete}
% A smooth manifold \textbf{\emph{with collared ends}} is
% a smooth $n$ manifold $X$ with a collection of topological embeddings
% $e _{i}: B _{i} \times (0,1) \to X $, $i \in I$, for
% topological spaces $B _{i}$ with the properties:
% \begin{enumerate}
% \item $$\mathcal{B} : = \{\image e _{i} ^{n}
% \,|\, i \in I, n \in \mathbb{N} - 0 \}$$ is a basis of end
% neighborhoods, where $e _{i} ^{n} = e _{i} | _{B _{i} \times
% (0, \frac{1}{n})}$.
% (That is each topological end neighborhood contains an element of $\mathcal{B}$.)
% \item  $X - \cup _{i} \image e _{i} ^{n _{i}}$ is compact for each sequence $\{n _{i}\} \subset \mathbb{N} -0$. 
% % \item  $\operatorname {core} (X) \subset  X$ is
% % a deformation retract of $X$.
% \end{enumerate}
% \end{definition}
\begin{terminology} From now on, all our metrics are
Riemann-Finsler (a.k.a. Finsler)
metrics unless specified to be Riemannian, and usually denoted by
just $g$. 
Completeness, always
means forward completeness, and it is an assumption for all
our metrics. Curvature always means sectional
curvature in the Riemannian case and flag curvature in the
Finsler case. Thus we will usually just
say complete metric $g$, for a forward complete
Riemann-Finsler metric.  	A reader may certainly choose to
interpret all metrics as Riemannian metrics, completeness
as standard completeness, and curvature as sectional
curvature. 
\end{terminology}
In what follows, $\pi_k(X)$ denotes the usual based homotopy group,
with basepoints suppressed when no confusion can arise.  We do not use
$\pi_1(X)$ to denote free homotopy classes.

\begin{definition}\label{definition_boundaryincompressible}
Let $X$ be a smooth manifold. Fix an exhaustion by
nested compact sets $\bigcup _{i \in \mathbb{N}} K _{i}
= X$, $K _{i} \supset K _{i-1} ^{\circ}$ for all $i \geq 1$. A free homotopy
class $\beta$ of loops in $X$ is called \textbf{\emph{end compressible}} if,
for every $i$, it has a representative contained in $X-K_i$.  We say that
$\beta$ is \textbf{\emph{end incompressible}} (or incompressible to the ends)
if it is not end compressible.
\end{definition}
Let $\pi _{1} ^{inc} (X)$ denote the
set of such end incompressible free homotopy classes. When $X$ is
compact, we set $\pi _{1} ^{inc} (X)$ to be the set of nonconstant free
homotopy classes of loops in $X$.

It is easily seen that the above is well defined
(independent of the choice of an exhaustion) and
moreover any homeomorphism $X _{1} \to X _{2}$ of
a pair of manifolds induces a set isomorphism $\pi _{1} ^{inc} (X _{1}) \to \pi _{1} ^{inc} (X _{2}) $.

Denote by $L _{\beta } X$ the class $\beta \in \pi _{1}
^{inc}(X) $ component of the free loop space of $X$, with
its compact open topology. 
Let $g$ be a complete metric on $X$, and let $S
(g, \beta) \subset L _{\beta } X$ denote the subspace of all
constant speed parametrized, smooth, closed $g$-geodesics in class $ \beta$.
\begin{definition} \label{definition_betataut}
We say that a metric $g$ on $X$ is \textbf{\emph{$\beta
$-taut}} if it is complete and $S (g, \beta )$ is compact. We will say that $g$ is \textbf{\emph{taut}} if
it is $\beta $-taut for each $\beta \in \pi _{1} ^{inc}
(X)$. 
\end{definition}
\begin{lemma} \label{lemma_nonpositiveistaut}
A complete metric $g$ with nonpositive curvature satisfies:
\begin{itemize}
	\item All of its closed geodesics are minimizing in their
	free homotopy class.
	\item It is taut.
\end{itemize}
\end{lemma}
\begin{proof} [Proof]
The first part is a standard consequence of the
Cartan-Hadamard theorem. The second part follows by the
first part and Corollary \ref{lemma_compact}.
\end{proof}
\begin{definition}\label{def_tauthomotopy} Let 
$ \beta \in \pi_{1} ^{inc}(X) $, and let
$g _{0}, g _{1}$
be a pair of $\beta $-taut metrics on $X$. 
A \textbf{\emph{$\beta $-taut deformation}} between $g _{0}, g _{1}$, is a continuous (in the topology of $C ^{0}$
convergence on compact sets) family $\{g _{t}\}$, $t \in
[0,1]$ of complete metrics on $X$, 
such that $$S (\{g _{t}\}, \beta ) := \{(o,t) \in L _{\beta }X
\times [0,1] \,|\, o \in S (g _{t}, \beta )\}$$ is compact.
We say that $\{g _{t}\}$ is a \textbf{\emph{taut deformation}}  if it is $\beta $-taut for
each $ \beta  \in \pi _{1} ^{inc}(X)$.  The above definitions
of tautness are extended naturally to the case of a smooth fibration $X
\hookrightarrow  P \to [0,1]$, with a smooth fiber-wise family of metrics.
\end{definition}
A useful criterion for $\beta $-tautness is the following.
\begin{theorem} [Proof in Section
\ref{sec_proofthmnotgeodesible}] \label{thm:notgeodesible}
Let $\beta \in \pi _{1} ^{inc} (X) $. Let $\{g _{t}\}
_{t \in [0,1]}$ be a continuous family of complete
metrics on $X$.
Suppose that:
\begin{equation*}
\sup _{t} |\sup _{o \in S  (g _{t}, \beta )} \ell _{g _{t}}
(o) -  \inf _{o \in S  (g _{t}, \beta )} \ell _{g _{t}} (o)| < \infty, 
\end{equation*}
where $\ell _{g _{t}}$ is the length functional
with respect to $g _{t}$,
then $\{g _{t}\}$ is $\beta $-taut.
% \end{equation*}
% Moreover, $S (\{g _{t}\}, \beta)$ is compact.
\end{theorem}
For example, the hypothesis is trivially satisfied if $g
_{t}$ have the property that all their class $\beta
$ closed geodesics are minimizing in their homotopy  class.
\begin{corollary} \label{cor_taut}
If $g _{t}$, $t \in [0,1]$ have nonpositive 
curvature then $\{g _{t}\}$ is taut.
\end{corollary}
\begin{proof} [Proof]
This follows by Theorem \ref{thm:notgeodesible} and by Lemma
\ref{lemma_nonpositiveistaut}. 
\end{proof}
% We give one such condition here by way of Theorem
% \ref{thm:notgeodesible}, which will follow shortly after
% some setup.
% \end{remark}

\subsubsection{The geodesic string counting invariant $F$} \label{sec_Definition of the invariant F}
Let $\mathcal{G} (X)$ be the set of equivalence classes of
taut metrics $g$, where $g _{0}$ is equivalent to $g _{1}$ whenever there is
a taut deformation between them.  We may denote an
equivalence class by its representative $g$ by a slight abuse
of notation. The following is proved in Section
\ref{sec:Definition of F}.

\begin{theorem} \label{thm:invariantF}
For each manifold $X$ there is a natural, nontrivial
functional: $$\operatorname {F}: \mathcal{G} (X) \times \pi
_{1} ^{inc} (X) \to \mathbb{Q}, $$  extending the definition
in \eqref{eq_Fgb}.
\end{theorem}
For $g$ $\beta$-regular and $\beta $-taut  we have already defined $F (g,\beta
)$ at the beginning of the Introduction. For a general
$\beta $-taut $g$
we define $F (g, \beta )$ in terms of the Fuller index. 
% The value $\operatorname {F} (g, \beta )  $ is a certain
% weighted count of the set of closed  $g$-geodesic
% strings in class $\beta $.  
% But one must take care of exactly how to count, as in general this set should be
% understood as an orbifold or rather a Kuranishi space (as
% introduced by Fakaya-Ono ~\cite{cite_FukayaOnoArnoldandGW}), hence this is why $\operatorname {F} $ is
% $\mathbb{Q} $ valued. 
\begin{corollary} \label{cor_existsSky} Suppose
for a pair $g _{0}, g _{1}$ of $\beta $-taut metrics on $X$:
\begin{equation*}
F (g _{0}, \beta ) \neq F (g _{1}, \beta),
\end{equation*}
then any path $\{g _{t}\}$, connecting $g _{0},
g _{1}$, is not $\beta$-taut.  
\end{corollary}
\begin{proof} [Proof]
The fact that any connecting $\{g _{t}\}$ is not $\beta $-taut is just a direct
corollary of the theorem above. 
\end{proof}
I was unable to find such a pair $g _{0}, g _{1}$, in fact
there is strong evidence to believe the following:
\begin{conjecture} \label{con_topological}
$F$ does not depend on the choice of a smooth structure and taut metric. In other
words we have the following. Let $X$ be a topological
manifold,
define:
$$\mathcal{S} (X): \pi _{1}  ^{inc} (X) \to \mathbb{Q} \sqcup \{\infty \}
$$ by:
$$\mathcal{S} (X) (\beta) = \begin{cases}
	\infty, &\text{ if $X$ does not admit a smooth structure
	and a $\beta $-taut Finsler metric. }\\
  F (g, \beta), &\text{ if $g$ is a $\beta$-taut Finsler
	metric on $X$. }	
\end{cases},$$
then $\mathcal{S} $ is well-defined and hence determines
a topological invariant of topological manifolds. So if
$f: X _{1} \to X _{2}$ is a homeomorphism then $\mathcal{S}
(X _{1}) (\beta ) = \mathcal{S} (X _{2}) (f _{*} (\beta) )$.
\end{conjecture}
Theorem \ref{theorem_Eulercharacteristic} in the following
section partially proves this conjecture.
\begin{remark} \label{rem_fibration_examples}
The smooth manifold version of this conjecture, i.e. that
$\mathcal{S} (X)$ is a smooth manifold invariant, is implied
by the Reeb sky catastrophe conjecture described in the next section. However, the
above seems to be much more basic as will be apparent from
the proof of Theorem \ref{theorem_Eulercharacteristic}. 
\end{remark}

% \begin{remark} \label{rem_Reeb}  The formulation of the
% conjecture naturally
% extends to Reeb vector fields. That is let $\lambda $ be
% a contact form on the unit cotangent bundle $C$ of a smooth
% manifold $X$ (for simplicity closed), with $\lambda
% $ representing the Liouville contact structure (i.e. the
% standard contact structure). Assume that
% the space of closed, class $\widetilde{\beta }  $ $\lambda
% $-Reeb orbits on $C$ is compact. Then the Fuller index (See
% Section \ref{sec:Definition of F}) of
% this compact set is a topological invariant of $X$. 
% \end{remark}
\subsubsection{Applications to sky catastrophes}
\label{sec_digression}  
\begin{definition} [Preliminary] \label{def_preliminarysky}
A \textbf{\emph{sky catastrophe}} for a smooth family $\{X
_{t} \}$, $t \in [0,1]$, of nonvanishing vector fields on
a closed manifold $M$ is a continuous family of closed orbit
strings $\tau \mapsto o _{t _{\tau} }$, $o _{t _{\tau}
} $ is an orbit string of $X _{t _{\tau} } $, $\tau \in [0,
\infty)$, such that the period of $o _{t _{\tau} } $ is unbounded from above.
\end{definition}
Sky catastrophes first appeared in the work of
Fuller ~\cite{cite_FullerBlueSky}.  
One general definition appears in
~\cite{cite_SavelyevFuller}, which we review in the
Appendix \ref{appendix_bluesky}, the main point of this definition is that
there are no regularity assumptions on $\{X _{t}\}$.
The preliminary definition becomes equivalent to the more
general definition given certain regularity
conditions on the family $\{X _{t}\}$. 

The following corollary of Theorem \ref{thm:notgeodesible} may be of independent interest.
\begin{corollary} \label{cor_sky_no_geodesible}
Sky catastrophes of vector fields on closed manifolds are
not geodesible by metrics all of whose geodesics are
minimizing. That is, if a family $\{X _{t}\}$ has a sky
catastrophe, there is no family $\{g _{t}\}$ of metrics (all
of whose geodesics are minimizing in their homotopy class) such
that the orbits of $X _{t}$ are unit speed parametrized $g _{t}$-geodesics.
\end{corollary}

Fuller at the end of ~\cite{cite_FullerIndex} has asked for
any metric conditions on vector fields to rule out sky
catastrophes. By the
above, nonpositivity of curvature is one such
condition. So this is a partial answer to his question.

% \begin{remark} \label{remark:geodesible}
The main conjecture in ~\cite{cite_SavelyevFuller} is that
geodesible, and more generally Reeb sky catastrophes ($\{X
_{t}\}$ are Reeb) are unstable, in the sense that there is
a $C ^{0}$ nearby family $\{X' _{t}\}$ which does not have
a sky catastrophe (at least in the fixed class $\beta $).
Let's call this \textbf{\emph{Reeb sky catastrophe
conjecture}}. The reason to conjecture such a thing, is that
a theorem in ~\cite{cite_SavelyevFuller} says that if $\{o
_{t _{\tau}}\} _{\tau}$ is a sky catastrophe for a family of
Reeb vector fields $\{X _{t}\}$ then the length of the
corresponding projection $\tau \mapsto t _{\tau}$ is infinite.
(Assuming some minor regularity on $\{o _{t _{\tau}}\}$.)
The latter suggests that the structure of such a sky
catastrophe is very fragile.
\begin{remark} \label{rem_sky}
The Reeb sky catastrophe conjecture implies the geodesible
Seifert conjecture (apparently still an open problem), by the main
result of ~\cite{cite_SavelyevFuller}. Hence, this is a subtle
question. 
% The qualitative structure of such
% potential geodesible or Reeb sky catastrophes is partially
% understood, ~\cite[Theorem 1.10]{cite_SavelyevFuller}.
% But
% this does not greatly aid constructing potential
% examples, which must be topologically/dynamically very complex, (there
% are necessarily infinitely many suitably synchronized bifurcation
% events).   
\end{remark}
The following is a stronger form of Corollary
\ref{cor_existsSky}.
\begin{theorem} \label{thm_sky_catastrophe} Suppose
for a pair $g _{0}, g _{1}$ of $\beta $-taut metrics on
a closed manifold $M$:
\begin{equation*}
F (g _{0}, \beta ) \neq F (g _{1}, \beta),
\end{equation*}
then any path $\{g _{t}\}$, connecting $g _{0},
g _{1}$, has a sky catastrophe, meaning that the
corresponding family $\{X _{t}\}$ of Reeb vector fields on
the unit cotangent bundle (cf. Section \ref{sec:Definition of F}) has a sky
catastrophe as in Definition \ref{def:bluesky}.
\end{theorem}

\begin{proof} [Proof]
This directly follows by \cite[Theorem 3.2]{cite_SavelyevFuller}.
\end{proof}
So that if such a pair $g _{0}, g _{1}$ exists,
the Reeb sky catastrophe conjecture is disproved and in
a rather extreme fashion. 

% If the answer to the question above is `yes', then we
% should be able to use the above corollary to find
% nonpositively
% curved metrics which cannot be joined by a continuous family of
% nonpositively curved metrics.
% (Apparently, existence of such metrics is open.)
% So Theorem \ref{thm:invariantF} can theoretically be used to
% prove existence of geodesible sky catastrophes, disproving
% the main conjecture of ~\cite{cite_SavelyevFuller}.
\subsubsection{Basic results on the invariant $F$} \label{sec_Basic results on the invariant $F$}

\begin{definition}\label{definition_indivisible}
Let $\beta$ be a free homotopy class of loops in $X$.  Choose a representative
of $\beta$ and a base point $x_0$ on its image.  This choice determines an
element
\[
        \beta_{x_0}\in \pi_1(X,x_0)
\]
well-defined up to inner automorphism.
\begin{itemize}
 \item An element $\gamma\in\pi_1(X,x_0)$ is \textbf{\emph{indivisible}} if
 whenever $\gamma=\alpha^k$ for some $\alpha\in\pi_1(X,x_0)$ and $k>0$, then
 $k=1$.
 \item An element $\gamma\in\pi_1(X,x_0)$ is a $k$-\textbf{\emph{power}} if
 $\gamma=\alpha^k$, $k>1$, for some $\alpha$ which is not an $n$-power for any
 $n>1$.
 \item An element $\gamma\in\pi_1(X,x_0)$ is
 \textbf{\emph{atomic}} if $\gamma = \alpha ^{k}$, for some $k \geq 1$, with $\alpha $  indivisible.
\end{itemize}
We say that the free homotopy class $\beta$ is indivisible, respectively a
$k$-power, respectively atomic, if $\beta_{x_0}$ has the corresponding property
for one, equivalently any, such choice of $x_0$ and representative.  This is
well-defined because the corresponding notions are invariant under inner
automorphism.  Equivalently, $\beta$ is a $k$-power precisely when it admits a
smooth representative with multiplicity $k$ whose underlying
simple representative is indivisible.
\end{definition}
Note that a class $\beta \in \pi _{1} ^{inc} (X)$ is
indivisible iff any representative of this class is not
multiply covered.
\begin{example} \label{exm_negativecurvatureFuller} Let $g$ be
a Riemannian metric with negative sectional curvature on a closed
manifold $X$ and $ \beta
\in \pi _{1} ^{inc} (X)$ a class represented by a multiplicity $n$ closed
geodesic, then 
\begin{equation} \label{equation_frac1n}
F (g, \beta ) = \frac{1}{n}.
\end{equation}
In particular,
if $\beta $ is indivisible then $F (g, \beta ) = {1}$. More generally,
\eqref{equation_frac1n} holds whenever $g$ has a unique
and nondegenerate geodesic string in class $\beta $, where
nondegenerate is as in the Introduction.
\end{example}
If $\beta \in \pi _{1} ^{inc} (X)$ is indivisible, then it
is easy to see that the reparametrization $S ^{1}$
action on $L _{\beta }X$ is free (Section \ref{sec_Proof of Theorem theorem_Eulercharacteristic}), so that  $H ^{S ^{1}} _{*} (L
_{\beta}X, \mathbb{Z}) \simeq H  _{*} (L
_{\beta}X/ {S ^{1}}, \mathbb{Z}) $, where $H ^{S ^{1}} _{*} (L
_{\beta}X, \mathbb{Z})$ denotes the $S ^{1}$-equivariant
homology. Moreover, we have the following result proved in
Section \ref{sec_Proof of Theorem theorem_Eulercharacteristic}.
\begin{theorem} \label{theorem_Eulercharacteristic}
Suppose that $\beta \in \pi _{1} ^{inc} (X)$ is indivisible,
and $X$ admits a $\beta $-taut metric. Then $H ^{S ^{1}} _{*} (L
_{\beta}X, \mathbb{Z})$ has finite total rank.
Denote by $\chi ^{S ^{1}} (L _{\beta} X)$ the Euler characteristic of this homology.
Then for any $\beta $-taut metric $g$ on $X$:
\begin{equation*}
F (g, \beta ) = \chi ^{S ^{1}}(L _{\beta }X).
\end{equation*}
Consequently, $F(g,\beta)$ is independent of the choice of $\beta$-taut metric $g$, for indivisible $\beta$.
Furthermore, Conjecture \ref{con_topological} holds on the
subset of classes $\beta $ which are indivisible.
\end{theorem}
At least for surfaces $\beta $-taut metrics (distinct from
nonpositive curvature metrics) can be
constructed explicitly, recall the example of Figure
\ref{fig:torus-beta-two-example}.
We can use the above, and the product formula of Theorem
\ref{thm:EulerProduct} to get:
\begin{theorem} [Proof in
Section \ref{sec:Proof of Theorem fibration}] \label{theorem_valuesOfInvariant} Every
rational number has the form $F (g, \beta )$ for some
$\beta$-taut Riemannian $g$ on some compact manifold $X$ and
for some $\beta $.
\end{theorem}

\subsubsection{Applications to existence of negative curvature
metrics} \label{sec_Applications to existence of negative curvature
metrics}
We are going to give a certain generalization of Preissmann's
theorem to fibrations, with possibly noncompact fibers,
also replacing the negative sectional curvature condition by
a significantly weaker condition. Though the definitions may
at first seem hard to understand, we just want enough
generality to incorporate interesting ``flat bundle''
examples (with anti-Anosov or periodic holonomy behavior). The latter are of interest in 3-manifold theory, for instance vis-à-vis Thurston's classification 
~\cite{cite_ThurstonClassificationSurfaceDiffeo}.

\begin{definition}\label{definition_fiberclass}
For a fiber bundle $Z \hookrightarrow X \to Y$, we say that
$\beta\in\pi_1^{\mathrm{inc}}(X)$ is a \textbf{\emph{fiber class}} if it lies
in the image of the set map
\[
        i_Z:\pi_1^{\mathrm{inc}}(Z)\to\pi_1^{\mathrm{inc}}(X)
\]
induced by inclusion of a fiber.
\end{definition}

% In what follows $\pi _{k} (X)$ for may denote both the based
% homotopy groups (at suitable base points) and sets of free
% homotopy classes. But we disambiguate by saying groups and 
% homomorphism when talking about groups.

\begin{lemma}\label{lem_fiber-representatives-holonomy-orbit}
Let
\[
        Z\hookrightarrow X\xrightarrow{p}Y
\]
be a smooth fiber bundle with connected fiber and connected base.  Assume
\[
        \pi_2(Y)=0.
\]
Fix $y_0\in Y$, and let
\[
        H=\pi_1(Y,y_0)
\]
act on $\pi_1^{\mathrm{inc}}(Z_{y_0})$ by holonomy.  Let
$\beta\in\pi_1^{\mathrm{inc}}(X)$ be a fiber class, and define
\[
        \mathcal R_{y_0}(\beta)
        =
        \left\{
        \alpha\in\pi_1^{\mathrm{inc}}(Z_{y_0})
        \;\middle|\;
        (i_{y_0})_*(\alpha)=\beta
        \right\}.
\]
Then $\mathcal R_{y_0}(\beta)$ is a single $H$-orbit.  In particular, if
\[
        \beta_Z,\beta_Z'\in\mathcal R_{y_0}(\beta),
\]
then
\[
        H\cdot\beta_Z=H\cdot\beta_Z'.
\]
Thus the holonomy orbit associated to the fiber class $\beta$ is independent
of the choice of fiber representative $\beta_Z$.
\end{lemma}

\begin{proof}
Choose basepoints $z_0\in Z_{y_0}$ and $x_0=i_{y_0}(z_0)$.  Since
$\pi_2(Y)=0$, the homotopy long exact sequence gives an exact sequence of
based groups
\[
        1\to \pi_1(Z_{y_0},z_0)
        \xrightarrow{(i_{y_0})_*}
        \pi_1(X,x_0)
        \xrightarrow{p_*}
        \pi_1(Y,y_0)
        \to 1.
\]
Thus we identify $\pi_1(Z_{y_0},z_0)$ with a normal subgroup of
$\pi_1(X,x_0)$, and the quotient is $\pi_1(Y,y_0)$.

Let $\alpha,\alpha'\in\mathcal R_{y_0}(\beta)$.  Choose based representatives
of these fiber free homotopy classes.  Since they determine the same total
free homotopy class $\beta$ in $X$, the corresponding elements of
$\pi_1(X,x_0)$ are conjugate.  Thus, after conjugating inside
$\pi_1(Z_{y_0},z_0)$ if necessary, there exists $x\in\pi_1(X,x_0)$ such that
\[
        \alpha'=x\alpha x^{-1}.
\]
Let
\[
        h=p_*(x)\in H.
\]
Conjugation by $x$ on the normal subgroup $\pi_1(Z_{y_0},z_0)$ induces the
holonomy action of $h$, up to an inner automorphism of
$\pi_1(Z_{y_0},z_0)$.  Passing to free homotopy classes in the fiber, we get
\[
        \alpha'=h\cdot\alpha.
\]
Thus all fiber representatives of the total free homotopy class $\beta$ belong
to a single holonomy orbit.

Conversely, if $\alpha'=h\cdot\alpha$, choose $x\in\pi_1(X,x_0)$ with
$p_*(x)=h$.  Then $x\alpha x^{-1}$ represents the same total free homotopy
class in $X$ as $\alpha$, and its fiber free homotopy class is $h\cdot\alpha$.
Hence the holonomy orbit is exactly the set of fiber representatives of
$\beta$.
\end{proof}
We assume from now on that $Z,Y$ are connected.
For a fiber class $\beta\in\pi_1^{\mathrm{inc}}(X)$, set
\[
        \mathcal R(\beta)
        =
        \left\{ \alpha \in \pi_1^{\mathrm{inc}}(Z) \,\middle|\,
        (i_{Z})_*(\alpha)=\beta \right\}.
\]
Assume that $\pi_2(Y)=0$, and put
\[
        H=\pi_1(Y,y_0).
\]
By Lemma~\ref{lem_fiber-representatives-holonomy-orbit}, the holonomy orbit
\[
        S=H\cdot\beta_Z
\]
is independent of the choice of the fiber representative
$\beta_Z\in\mathcal R(\beta)$.  If another fiber inclusion is chosen, the
corresponding holonomy actions are abstractly conjugate.

\begin{definition} \label{def_betafinite}
Let
\[
        Z\hookrightarrow X\xrightarrow{p}Y
\]
be a smooth fiber bundle, with $Y$ and $\pi_2(Y)=0$, and let
$\beta\in\pi_1^{\mathrm{inc}}(X)$ be a fiber class.  We say that $p$ is
\textbf{\emph{$\beta$-holonomy-finite}} if for one, equivalently any,
$\beta_Z\in\mathcal R(\beta)$, the orbit
\[
        S=\pi_1(Y,y_0)\cdot\beta_Z
\]
is finite.
\end{definition}
\begin{definition}
Let
\[
        Z\hookrightarrow X\xrightarrow{p}Y
\]
be a smooth fiber bundle, and let $g$ be a Riemannian metric on $X$.
Let
\[
        T^{\mathrm{vert}}X=\ker(dp)
\]
and let
\[
        \pi^{\mathrm{vert}}:TX\to T^{\mathrm{vert}}X
\]
be the $g$-orthogonal projection.  We say that $p$ is
$g$-\textbf{\emph{parallel}}  if
\[
        \nabla \pi^{\mathrm{vert}}=0.
\]
Equivalently,
\[
        \nabla_A(\pi^{\mathrm{vert}}B)
        =
        \pi^{\mathrm{vert}}(\nabla_A B)
\]
for all vector fields $A,B$ on $X$.
\end{definition}

\begin{definition}\label{def_periodic}
Let $\beta\in\pi_1^{\mathrm{inc}}(X)$ be a fiber class, for a smooth fibration
\[
        p:Z \hookrightarrow (X,g)\to(Y,g_Y).
\]
We say that this is a \textbf{\emph{$\beta$-periodic fibration}} if:
\begin{enumerate}
        \item $g$ is a complete Riemannian metric.
				\item $Z,Y$ are connected and $\pi _{2} (Y) =0 $. 
        \item $p$ is $g$-parallel and is a Riemannian
				submersion.
        \item $p$ is $\beta$-holonomy-finite.
        \item the fiber metrics $g_{y}$ on $Z _{y}$ are  taut. 
				% $\alpha$-taut for every
				%     \[
				%             \alpha\in H_{y_0}\cdot\beta_Z.
				%     \]
\end{enumerate}
\end{definition}
The metrics $g,g _{Y}$ may be kept implicit.
In the above definition of a $\beta $-periodic fibration and the hypothesis of the following theorem we need
the auxiliary metric $g$ on $X$ to be Riemannian, and there is no obvious
extension of the theorem to the Riemann-Finsler case.
However, the conclusions of the theorem are for
Riemann-Finsler metrics. The following is perhaps the main
technical result of the paper, proved in Section
\ref{sec:Proof of Theorem fibration}.
\begin{theorem} \label{thm:Fibration}   Let $p: Z \hookrightarrow (X,g) \to
(Y, g _{Y})$ be a $\beta$-periodic fibration, where $ \beta \in
\pi_{1} ^{inc}(X)$ is a fiber class.
Suppose further that $Y$ is connected, 
closed, $\chi(Y) \neq \pm 1$ and is such that all smooth closed contractible $g _{Y}$-geodesics in $Y$ are constant. Then the following holds:
\begin{itemize}
	% \item  If $|\chi (Y)| > \operatorname {lcm} (1, \ldots, k)$,  where
	% $\operatorname {lcm} (1, \ldots, k)$ denotes the least
	% common multiple of $1, \ldots, k$, then $g$ cannot be
	% $\beta $-taut deformed to a Riemann-Finsler metric $g'$
	% on $X$ with a unique nondegenerate geodesic string in class $\beta$, in particular to a complete metric with negative curvature.  \label{part_1}
% \item If $\chi (Y) = 0$ then $F (g, \beta) =0$ and $X$ does not admit a complete Riemann-Finsler metric with negative curvature.
% Moreover, if $\beta$ is at most a $k$-power then $X$ does
% not admit a Riemann-Finsler metric with a unique and
% nondegenerate class $\beta $ geodesic string. \label{part_chi=0}
\item $X$ does not admit a complete Riemannian metric with negative curvature. 
\item $X$ does not admit
a forward complete Riemann-Finsler metric with
a unique nondegenerate class $\beta$ geodesic string. \label{part_chipm1}

\end{itemize} 
\end{theorem}
Note that $\chi(Y) \neq 1$ is of course essential, as  the
trivial fibration $X \to \{pt\}$, with $X$ admitting
a complete negatively curved metric, will satisfy the
hypothesis. The condition that there is a fiber class $\beta
\in \pi _{1} ^{inc} (X)$ is also essential; as any vector
bundle over a manifold admitting a Riemannian metric of
negative curvature also admits a metric of negative curvature, Anderson
~\cite{cite_AndersonNegativeCurvaturefibration}.

\begin{corollary} [Proof in Section \ref{sec:Proof of Theorem fibration}] \label{corollary_product}
Let $Y,Z$ be smooth manifolds, let $\beta_Z\in\pi_1^{\mathrm{inc}}(Z)$, and
assume that $\chi(Y)\neq\pm1$.  Suppose that $Y$ is closed and admits a metric
all of whose contractible geodesics are constant, for instance a nonpositive
curvature Riemannian metric, and suppose that $Z$ admits
a taut metric.  Let
\[
        X=Y\times Z
\]
with projection $X\to Y$, let $i:Z\to X$ be a fiber inclusion, and set
\[
        \beta=i_*(\beta_Z).
\]
Then:
\begin{itemize}
\item $X$ does not admit a complete Riemannian metric of negative sectional
curvature.
\item $X$ does not admit a forward complete Riemann-Finsler metric with a
unique and nondegenerate class $\beta$ geodesic string.
\end{itemize}
\end{corollary}
\begin{example} \label{exm_} Take $Y$ any nonpositively
curved closed manifold and $Z$ a surface of genus at least
one, possibly infinite type.   
\end{example}
% As I understand, a weaker statement that $X$ does not admit
% a geometrically finite complete Riemannian metric of
% negative sectional curvature, follows by Preissmann's theorem
% ~\cite{cite_PreissmanNegativeCurvatureProducts} on Abelian subgroups
% of $\pi _{1} $ of a  negatively curved manifold.
This theorem is very close to being sharp, for example the
conclusion of the corollary is false if $Y=S ^1$ and $Z
= S ^1 \times \mathbb{R}  $. As $X = T ^{2} \times
\mathbb{R} ^{} $ admits the warped
product metric $g _{\mathbb{R}}  \times _{e ^{t}} g _{T ^{2}} $ (with respect to the function $f= e ^{t}$ on
$\mathbb{R} ^{} $) where $g _{T ^{2}}, g _{\mathbb{R} ^{} }$
are the flat metrics. This warped product has constant negative sectional curvature $-1$.
So it is essential that not only $\pi _{1} (Z)$ be
nontrivial but also that there is a class incompressible to the
ends. Of course, $\chi(Y) \neq 1$ is also
obviously essential, otherwise we may take $Y =pt$. The condition $\chi(Y) \neq -1$ is however
not obviously essential.

A further basic set of examples for the theorem is obtained
by starting with a Riemannian manifold $(Y, g _{Y})$, and any homomorphism 
\begin{equation} \label{equation_phi}
\phi: \pi _{1} (Y, y _{0}) \to \operatorname
{Isom} (Z, g _{Z}),  \quad \text{(the group of all isometries).} 
\end{equation}
where $g _{Z}$ is a taut Riemannian metric and $\beta _{Z}
\in \pi _{1} ^{inc} (Z) $. 
(For example $(Z,g _{Z})$ is a non-simply
connected complete hyperbolic surface of genus at least one).  Suppose further:
\begin{enumerate}
\item  The orbit 
\[
        \operatorname{Orb}(\beta_Z)
        =
        \left\{
        \phi_*(\gamma)(\beta_Z)
        \;\middle|\;
        \gamma\in\pi_1(Y,y_0)
        \right\}
\]
is finite.
\label{item_finiteorbit}
\item $Y$ is closed and connected.
\item All contractible smooth closed $g _{Y}$ geodesics in $Y$ are constant.
\end{enumerate}
We have the obvious induced diagonal action $$\pi
_{1} (Y, y _{0}) \to \operatorname {Diff}  (Z \times
\widetilde{Y}), \text{ (the group of all
diffeomorphisms)},$$ $$\gamma \mapsto ((z, y) \mapsto (\phi
(\gamma) (z), \gamma \cdot y)), $$
for $\widetilde{Y} $ the universal
cover of $Y$.  Taking
the quotient of $Z \times \widetilde{Y} $ by this action, we get an associated ``flat'' bundle $Z
\hookrightarrow X _{\phi} \xrightarrow{p} Y$, with a metric
$g _{\phi }$
induced from the product metric $\widetilde{g} = g _{Z}
\oplus g _{Y}$, on the covering space $q: Z \times
\widetilde{Y} \to Z \times Y$.  The natural projection $p:
X _{\phi } \to Y$ is then a Riemannian submersion and is
a $g _{\phi}$-parallel. The condition on the finiteness of
$\operatorname {Orb} $ implies that $p$ is $\beta $-holonomy
finite where $\beta $ is the image by $(i _{Z})_*: \pi _{1}
^{inc} (Z) \to \pi _{1} ^{inc} (X _{\phi }) $ of $\beta
_{Z}$. 
It follows that $p$ is a $\beta $-periodic fibration.
% \begin{corollary} [Corollary of Theorem \ref{thm:Fibration}] \label{lemma_betataut}
% Let $p: (X _{\phi}, g _{\phi }) \to (Y, g _{Y})$  be as above, then this is a $\beta $-periodic fibration.  
% \end{corollary}
Hence $p: (X _{\phi}, g _{\phi }) \to (Y, g _{Y})$
satisfies the hypothesis of Theorem \ref{thm:Fibration}.
\begin{example} \label{example_phi}
Suppose we have $\beta _{Z} \in \pi_{1} ^{inc}({Z})$, and 
let $\phi: Z \to Z$ be an isometry of a 
taut metric $g _{Z}$, such that 
$Orb$ is finite.
Then for $\beta = (i _{Z}) _* (\beta _{Z})$, by the construction above, the
mapping torus $$(Z, g _{Z}) \hookrightarrow  (X _{\phi},
g _{\phi})  \xrightarrow{\pi
_{} } S ^{1}$$  has the structure of a $\beta$-periodic
fibration,
satisfying the hypothesis of Theorem \ref{thm:Fibration}.  
\end{example}
The example can be made more concrete as follows.
\begin{lemma}\label{lem:isometry-finite-orbits-finite-type}
Let $Z$ be a complete Riemannian manifold of finite type.  Then every
isometry of $Z$ has finite orbits on $\pi_1^{\mathrm{inc}}(Z).$ 
\end{lemma}
\begin{proof}
We first prove that, for every $L>0$, only finitely many
classes in $        \pi_1^{\mathrm{inc}}(Z)$
have representatives of length at most $L$.
Choose a compact core
\[
        K_0\subset Z
\]
such that $Z-K_0$ is a finite union of collar ends.  Since $Z$ is complete,
the closed $L$-neighborhood
\[
        K_L=N_L(K_0)
\]
is compact.  
If a representative of an end-incompressible class lies entirely in
$Z-K_0$, then it lies in a collar end and can be pushed arbitrarily far out,
contradicting end-incompressibility.  Therefore every representative of an
end-incompressible class must meet $K_0$.  If such a representative has
length at most $L$, then its image lies in $K_L=N_L(K_0)$.

It remains to use compactness.  A compact Riemannian manifold, or compact
domain after a harmless thickening of $K_L$, contains only finitely many free
homotopy classes with representatives of length at most $L$.  
For instance,
choose a finite good cover with Lebesgue number $\delta>0$.  A loop of length
at most $L$ is encoded by an edge loop in the nerve of the
cover, of combinatorial length
bounded by $N (L, \delta )$, and there are only finitely
many such combinatorial loops.  Hence
there are only finitely many end-incompressible classes represented by loops of length at most $L$.

Now let $\phi:Z\to Z$ be an isometry and let
\[
        \beta_Z\in\pi_1^{\mathrm{inc}}(Z).
\]
Choose a loop $c$ representing $\beta_Z$, and set
\[
        L=\ell(c).
\]
Then $\phi^k(c)$ represents $\phi_*^k(\beta_Z)$ and has the same length
$L$ for every $k\in\mathbb Z$.  Since $\phi$ preserves
end-incompressibility, the orbit
\[
        \{\phi_*^k(\beta_Z)\mid k\in\mathbb Z\}
\]
is contained in the finite set of end-incompressible classes represented by
loops of length at most $L$.  Therefore the orbit is finite.
\end{proof}
\begin{corollary} \label{corollary_Thurston0}
Let $$(Z ,{g_ Z}) \hookrightarrow (X _{\phi}, g _{\phi}) \to
S ^{1}$$ be the mapping torus of an isometry of 
$Z, g _{Z}$, with $Z$  finite type and $g _{Z}$ taut. Let $ \beta \in \pi_{1} ^{inc}(X _{\phi })$ be
the image of some $\beta _{Z} \in \pi _{1} ^{inc} (Z)$ under fiber inclusion, as previously.
\begin{enumerate}
\item $X _{\phi }$ does not admit a complete Riemannian  metric with negative curvature.
\item $X _{\phi}$ does not admit a forward complete Riemann-Finsler metric with
a unique nondegenerate class $\beta$ geodesic string. 
\item If $\beta $ is indivisible, $\chi ^{S ^{1}}(L _{\beta}
X _{\phi } ) =0$ and so $X _{\phi}$ does not admit 
a $\beta $-regular forward complete Finsler metric with an odd number of $\beta $-class geodesic strings. 
\end{enumerate}
\end{corollary}
% In the special case when $Z$ is compact, the first part of
% the above corollary can be deduced from
% Preissmann's theorem, see also ~\cite[Theorem
% 9.3.4]{cite_BuragoLengthSpaces} for a generalization that
% fits our Finsler setting. 

The following questions seem to be open.
\begin{question} \label{que_holonomy_finite_auto} Is the 
finite type assumption necessary for Corollary
\ref{corollary_Thurston0} in the case of surfaces? The
answer is likely no, as the intuition of Nielsen-Thurston classification (it cannot be
directly applied in the infinite type setting) suggests that
$\phi $ should have pseudo-Anosov type behavior which an
isometry even on an infinite type surface should not have.
\end{question}

% \begin{comment}
% \subsection{Estimated counts of multiply covered
% geodesics} \label{sec_Applications to counts of multiply covered
% geodesics}
% The following is a more extensive version of Theorem
% \ref{thm_IntroTn}.
% \begin{theorem} [Proof in Section \ref{sec_Proof of Theorem thm_estimategeodesics}]\label{thm_estimateIntro}
% Suppose that $g _{1}$ is a taut metric on $X$, taut
% deformation equivalent to a complete metric of negative
% curvature (everything is Finsler). 
% Suppose that $\beta \in \pi _{1} ^{inc} (X)$ is a $k$-power.
% Let $L _{\beta}$ be the length of a class $\beta$, $g _{1}$-geodesic.
% For all $\epsilon >0$
% sufficiently small, whenever $g'$ is $C
% ^{2}$ $\epsilon$ close to $g _{1}$, and is $\beta $-regular,
% we have:
% \begin{equation*} 
% \displaystyle \sum_{o \in
% \mathcal O _{2  L _{\beta}} (g',\beta)} \frac{(-1) ^{\operatorname {morse}  (o)}}
% {\mult (o)} = \frac{1}{k},  
% \end {equation*}
% where $\mathcal O _{2  L _{\beta}} (g', \beta)$ is the set of class $\beta
% $ geodesic strings with $g'$-length less than $2  L _{\beta}$.
% \end{theorem}
% \begin{corollary} \label{cor_estimate}
% Let $X,g$ be a complete Finsler manifold with finitely many
% class $\beta \in \pi _{1}  ^{inc} (X)$, geodesic strings.  Suppose that $\beta$ is a $p$-power for $p$ a prime. 
% Then whenever $g$ has  a class $\beta$ geodesic string with multiplicity $k \neq p$, then $g$ has at least two such geodesic strings.
% \end{corollary}
\section{Compactness preliminaries} \label{sec_Compactness preliminaries}
\begin{lemma} \label{lemma_containedinCompact} Suppose that
$g$ is a complete metric on $X$, $\beta \in \pi _{1}  ^{inc}
(X)$ and let $S \subset L _{\beta} X$ be a subset on which
the $g$-length functional is bounded from above. Then the images in $X$ of
elements of $S$ are contained in a fixed compact subset of $X$.
\end{lemma}
\begin{proof} [Proof]
Suppose otherwise. Fix an exhaustion by
nested compact sets $$\bigcup _{i \in \mathbb{N} } K _{i}
= X, \quad K ^{\circ} _{i} \supset
K _{i-1}.$$ 
Then either there is sequence $\{o _{i}\} _{i \in \mathbb{N}
}$, $o _{i} \in S$  such that $o _{i} \in K _{i} ^{c} $, for $K _{i} ^{c}$ the
complement of $K _{i}$, which contradicts the fact that
$\beta $ is end incompressible. 
Or there is a sequence $\{o _{k} \} _{k \in \mathbb{N} }$,
$o _{k} \in S$ such that:
\begin{enumerate}
\item Each $o _{k}$
intersects $K _{i _{0}}$ for some $i _{0}$ fixed.
\label{condition_intersects}
\item For each $i \in \mathbb{N} $
there is a $k _{i} >i$ such that $o _{k _{i}}$ is
not contained in $K _{i}$. \label{condition_notcontained}
\end{enumerate}
Now if $\operatorname {diam} (o _{k})$ is bounded in $k$, then condition
\ref{condition_intersects} implies that $o _{k}$ are
contained in a set of bounded diameter. (Here $\operatorname
{diam} (o _{k})$ denotes the diameter of $\image o _{k}$.) Consequently, by
Hopf-Rinow theorem ~\cite{cite_ChernBook},  $o _{k}$ are contained in a compact set.
But this contradicts condition \ref{condition_notcontained},
and the fact that $K _{i}$ form an exhaustion of $X$. 

Thus, we conclude that $\operatorname {diam} (o _{k})$ is
unbounded, but this contradicts the hypothesis.
\end{proof}
\begin{lemma} \label{lem_SA} 
Let $g$ be a complete metric on $X$, and let
$\beta\in \pi_1^{\mathrm{inc}}(X)$. Denote by $S _{A} (g,
\beta ) \subset S (g, \beta) $ the subset of elements with
$g$-length at most $A$. Then $S _{A} (g,
\beta ) \subset S (g, \beta ) $ is compact.
\end{lemma}
\begin{proof} [Proof] By Lemma
\ref{lemma_containedinCompact}, the images of elements of $S _{A} (g,
\beta ) \subset S (g, \beta ) $ are contained in compact
$K$. Since elements of $S _{A} (g,
\beta ) \subset S (g, \beta ) $ are parametrized by constant
speed, and since we have a length bound $S _{A} (g,
\beta )$  is an equicontinuous collection. The result then
follows by the Arzelà--Ascoli theorem and standard arguments. 
\end{proof}
\begin{corollary} \label{lemma_compact}
Suppose that $g$ is a complete metric on $X$, all
of whose class $\beta \in \pi _{1} ^{inc} (X)$ geodesics are
minimizing, then $g$ is $\beta$-taut.
\end{corollary}

\begin{lemma} \label{lem_existsgeod}
Let $g$ be a complete metric on $X$, and let
$\beta\in \pi_1^{\mathrm{inc}}(X)$. Then there exists a closed
$g$-geodesic minimizing length among all loops in the free homotopy class
$\beta$.
\end{lemma}

\begin{proof}
Suppose otherwise and
let
\[
        A=\inf_{\gamma\in L_\beta X}\ell_g(\gamma).
\]
By the hypothesis that there is no geodesic with length $A$, 
and by Lemma \ref{lem_SA}, there is 
an $\epsilon >0$ such that there is no closed geodesic $o$
with $A \leq \ell_{g} (o) < A + \epsilon$.  

By Lemma \ref{lemma_containedinCompact}
the subset $M _{\beta, A} \subset L _{\beta}X$ consisting of
elements with length less than $A+ \epsilon$ is contained in a compact $K$. 
Pick
$o \in M _{\beta,A}$ and apply the curve shortening
process to it. Let $o _{k}$, $k \in \mathbb{N} $  be the stages of this process.
Then $\image o _{k} \subset K$. In particular $\{o _{k}\}$
must converge to a geodesic, for
instance ~\cite{cite_MilnorMorsetheory}, (we can also argue
via Arzelà--Ascoli to prove convergence)
but this contradicts the condition on $\epsilon $. 
\end{proof}

\subsection{Proof of Theorem \ref{thm:notgeodesible}}
\label{sec_proofthmnotgeodesible}
Let $\{g _{t}\}$, $t \in [0,1]$  be as in the
hypothesis, with 
\begin{equation} \label{eq_mainC}
\sup _{t} |\sup _{o \in S  (g _{t}, \beta )} \ell _{g _{t}}
(o) -  \inf
_{o \in S  (g _{t}, \beta )} \ell _{g _{t}} (o)| < C, 
\end{equation}
and suppose that 
\begin{equation*}
\sup _{(o,t) \in \mathcal S(\{g _{t}\},\beta )} \ell _{g _{t}} (o) = \infty.
\end{equation*}
Then we have a sequence $\{o _{k}\}$, $k \in \mathbb{N}
$, of closed class $\beta $ $g _{t _{k}}$-geodesics
in $X$, satisfying:
\begin{enumerate} 
	\item $\lim
_{k \to \infty } t _{k} = t _{\infty } \in [0,1]. $
\label{cond:1}
\item  $\lim _{k \to \infty }\ell_{g _{t _{k}}} (o _{k})
= \infty $,  where $\ell_{g _{t _{k}}} (o _{t _{k}})$ is
the length with respect to $g _{t _{k}}$. \label{cond:2}
\end{enumerate}
Applying Lemma \ref{lem_existsgeod}, we may choose a length-minimizing closed
$g_{t_\infty}$-geodesic $o_\infty$ in the class $\beta$.
And let $L$ denote
its $g _{t_{\infty} }$-length. Let $g _{aux}$ be
a fixed auxiliary metric on $X$, and let $L _{aux} $ be the
$g _{aux }$ length of $o _{\infty }$.

Define a pseudo-metric $d _{C^0}$ on the space of metrics on
$X$ as follows. Set $K = \image o _{\infty}$. And
set $$V = \{v \in TX \,|\, \pi ({v})  \in
K \text{ for $\pi _{}: TX \to X$ the canonical
projection, and $|v| _{aux} = 1$} \},$$
where $|v| _{aux}$ is the norm taken with
respect to $g _{aux}$.

Then define: $$d _{C ^{0}} (g _{1}, g _{2}) = \sup _{v \in
V} | |v| _{g _{1}} - |v| _{g _{2}}|. $$  

By Properties \ref{cond:1} and \ref{cond:2}, for every
$\epsilon >0$  we may find
a  $k > 0$ such that: 
\begin{equation} \label{eq:lessepsilon}
d _{C ^{0}} (g _{t
_{k}}, g _{t _{\infty }}) < \epsilon 
\end{equation}
and 
\begin{equation} \label{eq_Laux}
\ell_{g _{t _{k}}}
(o _{k}) > C + L + L _{aux} \cdot \epsilon.
\end{equation}
% By \eqref{eq:lessepsilon} and	\eqref{eq_Laux},  

By \eqref{eq:lessepsilon},
we have:
% \begin{equation*}
% |\ell_{g _{t _{k}}} (o _{k}) - \ell_{g _{t _{k}}} (o _{\infty})|
% \geq |C + L + L _{aux} \cdot \epsilon - \ell_{g _{t _{k}}} (o
% _{\infty})|
% \end{equation*}
$$\ell_{g _{t _{k}}} (o
_{\infty}) < \ell_{g _{t _{\infty }}} (o _{\infty}) + L _{aux}
\cdot \epsilon = L + L _{aux} \cdot \epsilon.$$
Combining with \eqref{eq_Laux} we get:
\begin{equation*}
\ell_{g _{t _{k}}}
(o _{k}) > \ell_{g _{t _{k}}} (o _{\infty }) +C.
\end{equation*}
Since we may find a closed $g _{t _{k}}$-geodesic $o'$
satisfying $\ell_{g _{t _{k}}} (o') \leq \ell_{g _{t _{k}}} (o
_{\infty })$,
we get that
$$|\max _{o \in S  (g _{t _{k}}, \beta)} \ell_{g _{t _{k}}} (o)
-  \min _{o \in S  (g _{t _{k}}, \beta)} \ell_{g _{t _{k}}} (o)|> C,
$$
and so we are in contradiction. 

Thus, \begin{equation*}
\sup _{(o,t) \in \mathcal O(\{g _{t}\},\beta )} \ell_{g
_{t}} (o) < \infty.
\end{equation*}
Then compactness of $\mathcal S(\{g _{t}\},\beta )$
follows by a slight modification of the proof of Lemma \ref{lem_SA}. 
\qed

\section{Preliminaries on Reeb flow} \label{sec_Some preliminaries on Reeb dynamics}
Let $(C ^{2n+1}, \lambda ) $ be a contact manifold with
$\lambda$ a contact form, that is a one form such that $\lambda
\wedge (d \lambda) ^{n} \neq 0$.    Denote by
$R^{\lambda} $ the Reeb vector field 
satisfying: $$ d\lambda (R^{\lambda}, \cdot ) = 0, 
\quad \lambda (R^{\lambda}) = 1.$$ 
Recall that a \textbf{\emph{closed $\lambda $-Reeb orbit}}
(or just Reeb orbit when $\lambda $ is implicit)  is
a smooth map $$o: (S ^{1} = \mathbb{R} / \mathbb{Z})    \to
C $$ such 
that $$ \dot o (t) = c R^{\lambda}  (o (t)), $$ 
with $\dot o (t) $ denoting the time derivative,
for some $c>0$ called the period. Let $S(R^{\lambda},\beta)$ denote the space of all closed ${\lambda }$-Reeb
orbits in free homotopy class $\beta $, with its compact
open topology, and set $$\mathcal O(R^{\lambda},\beta) = S(R^{\lambda},\beta)/S ^{1}, $$
where $S ^{1} = \mathbb{R} ^{} /\mathbb{Z}  $ acts by reparametrization $t \cdot o(\tau)
= o (t + \tau)$.

% The period map $A: S(R^{\lambda},\beta) \to \mathbb{R} ^{}  $,
% is the map $o \mapsto \int _{S ^{1}} o ^{*} \lambda $. This
% induces a map $A: \mathcal O(R^{\lambda}, \beta ) \to
% \mathbb{R} ^{} $, likewise called the period.

\section{\texorpdfstring{Definition of the functional $F$ and proofs of auxiliary results}{Definition of the functional F and proofs of auxiliary results}} \label{sec:Definition of F}
Let $X$ be a manifold. Fix once and for all an auxiliary Riemannian metric
$g_0$, and set
\[
        C=S^*_{g_0}X.
\]
For any Riemann-Finsler metric $g$, let $F_g^*$ denote the dual norm, and let
\[
        \rho_g:C\to S_g^*X,
        \qquad
        \rho_g(\xi)=\frac{\xi}{F_g^*(\xi)}
\]
be the fiberwise radial projection.  We define a contact form on the fixed
manifold $C$ by
\[
        \lambda_g=\rho_g^*\lambda_{\mathrm{can}},
\]
where $\lambda_{\mathrm{can}}$ is the Liouville one-form on $T^*X$.  We write
\[
        R^{\lambda_g}
\]
for the Reeb vector field of $\lambda_g$ on this fixed manifold $C$.  Under
$\rho_g$, the flow of $R^{\lambda_g}$ corresponds to
the $g$-geodesic flow.

The same convention applies to the Sasaki metric.  If $g$ is Riemannian, let
$g_S^{\mathrm{nat}}$ denote the usual Sasaki metric on $S_g^*X$, and set
\[
        g_S=\rho_g^*g_S^{\mathrm{nat}}
\]
on the fixed manifold $C$.    Similarly, all metric-dependent functions on $S_g^*X$, such
as verticality functions, are regarded as functions on $C$ by pullback through
$\rho_g$.  This fixed-$C$ convention will be used throughout.

% \subsection{Lift classes in the cosphere bundle}\label{subsec:lift-classes}
%
%
% We record the mechanism used to pass from a class of curves in $X$ to a
% free homotopy class in the cosphere bundle.  

\subsection{Lifted geodesic strings}\label{subsec:lifted-geodesic-strings}
We recall how closed geodesics in $X$ are lifted to closed Reeb orbits in
its cosphere bundle.  Let $C$ be the fixed cosphere model used in this paper,
and let
\[
        q:C\to X
\]
denote the projection.  For a Riemann-Finsler metric $g$, the contact form
$\lambda_g$ on $C$ has Reeb vector field $R^{\lambda_g}$.  Its closed orbits
project to oriented closed constant-speed $g$-geodesics in $X$.

Conversely, every oriented closed constant-speed $g$-geodesic has a unique
Reeb lift to $C$.  In the Riemannian case, if
\[
        o:S^1\to X
\]
is a nonconstant closed $g$-geodesic, its lift to the $g$-unit cotangent bundle
is
\[
        \widetilde o_g(t)
        =
        \frac{g_{o(t)}(\dot o(t),\,\cdot)}
             {|\dot o(t)|_g}
        \in S_g^*X .
\]
Since $o$ is a geodesic, this covector field is parallel along $o$.  Thus,
under the Levi--Civita splitting of $T(T^*X)$, the curve $\widetilde o_g$ is
horizontal.  Hence, after the constant reparametrization determined by
the speed of $o$, it is a trajectory of the Reeb vector field.  Passing to the
fixed model $C$ by radial projection gives the corresponding closed
$R^{\lambda_g}$-orbit in $C$.  The same statement holds in the
Riemann-Finsler case using the Legendre transform.

The projection $q$ induces a homeomorphism between closed $R^{\lambda_g}$-orbit strings
in $C$ and oriented closed constant-speed $g$-geodesic strings in $X$.  Hence, for every class
$\beta\in\pi_1^{\mathrm{inc}}(X)$, if $\mathcal O(g,\beta)$
denotes the class $\beta$ $g$-geodesic string set as usual, then its Reeb lift
\[
        \widetilde{\mathcal O}(g,\beta)
        =
        \left\{
        \widetilde o
        \ \middle|\
        o\in\mathcal O(g,\beta)
        \right\}
        \subset
        \mathcal O(R^{\lambda_g})
\]
is homeomorphic to $\mathcal O(g,\beta)$.  In particular, if
$\mathcal O(g,\beta)$ is compact, then
$\widetilde{\mathcal O}(g,\beta)$ is compact.
Moreover, $\widetilde{\mathcal O}(g,\beta)$ is open in
$\mathcal O(R^{\lambda_g})$, since the projected free homotopy class in
$X$ is locally constant on the loop space of $C$.
\subsection {Definition via the Fuller index}
Suppose then that $g$ is $\beta$-taut.  Define 
$$
        F(g,\beta) = i(\widetilde{\mathcal
				O}(g,\beta), R ^{\lambda _{g}}) \in \mathbb Q,
$$
where the right hand side is the Fuller index, see Appendix~\ref{appendix:Fuller}.  

For a path of metrics $\{g_t\}_{t\in[0,1]}$,
the lifted cobordism of orbit string spaces
\[
        \widetilde{\mathcal O}(\{g_t\},\beta)
        =
        \left\{
        (\widetilde o, t)
        \ \middle|\
        o\in\mathcal O(g_t,\beta)
        \right\}
\]
is likewise homeomorphic to the downstairs cobordism
$\mathcal O(\{g_t\},\beta)$.  

\begin{proof} [Proof of Theorem \ref{thm:invariantF}]
Let $\beta \in \pi _{1} ^{inc} (X)$ be given, and let
$\{g_t\}_{t\in[0,1]}$ be a $\beta $-taut deformation.
As discussed above, the lifted cobordism
$\widetilde{\mathcal O}(\{g_t\},\beta)$ is compact by
$\beta$-tautness.  It is also open in the orbit cobordism, since the projected free homotopy class in $X$ is locally constant.  
Basic
invariance of the Fuller index, namely
\eqref{eq_basicinvariance}, then immediately gives 
\[
        F(g_0,\beta)=F(g_1,\beta).
\]
\end{proof}
% \section{Proof of Theorem \ref{thm_estimateIntro}} \label{sec_Proof of Theorem thm_estimategeodesics}
% We already know by Example
% \ref{exm_negativecurvatureFuller} that $F (g _{0}, \beta)
% = \frac{1}{k}$ and hence by Theorem \ref{thm:invariantF}
% $F (g _{1}, \beta) = \frac{1}{k}$.
% Let $U$ denote the open subset of $L _{\beta} X$ consisting
% of loops with $g _{1}$-length less then $2  L _{\beta}$.
% By ~\cite[Lemma 4.1]{cite_SavelyevFuller}, 
% for all $\epsilon >0$ sufficiently small, for any $g'$, $C
% ^{2}$ $\epsilon $-close to $g$ the following holds.
% Set $g' _{t} = (1-t) \cdot g _{1} + t \cdot g'$, for
% $t \in [0,1]$, then $$ \widetilde{N} = \mathcal O(\{g'
% _{t}\}, \beta) \cap (U \times [0,1])$$ is an open and compact subset of $\mathcal O(\{g' _{t}\}, \beta)$. 	
%
% Now set $$N_1 = \widetilde{N} \cap (L _{\beta} X \times \{1\}), $$ and $N_0= \mathcal O
% (g _{1}, \beta)$. 
% By the invariance property \eqref{eq_basicinvariance} of the Fuller index, we then have
% that $$\frac{1}{k} = i (N_0, R ^{\lambda _{g _{1}}})
% = i (N_1, R ^{\lambda _{g'}}). $$
% On the other hand, by construction and by index computations
% as in ~\cite[Section 2]{cite_SavelyevFuller}),
% we get:
% $$i (N_1, R ^{\lambda _{g'}}) =  \displaystyle \sum_{o \in
% \mathcal O(g',\beta) \cap U} \frac{(-1) ^{\operatorname {morse}  (o)}} {\mult (o)}.$$
% If $\epsilon $ is chosen to be sufficiently small then
% $\mathcal O(g',\beta) \cap U = \mathcal O _{2  L _{\beta}} (g', \beta
% )$. So that we are done. 
% \qed
\section{Proof of Theorem \ref{theorem_Eulercharacteristic}}
\label{sec_Proof of Theorem theorem_Eulercharacteristic}
For simplicity, as the full argument is already lengthy, we prove the theorem under the assumption that $g$ is Riemannian. In
the Finsler case we must use a more modern variational setup
for the Hilbert manifold of $H^1$-loops.  In the Finsler
case the energy functional is generally only $C^1$, rather than $C^2$, on the
Hilbert loop space.  Nevertheless the required Palais--Smale compactness and
critical point theory for closed Finsler geodesics are standard; see
Caponio--Javaloyes--Masiello
\cite{cite_CaponioJavaloyesMasiello2011},
Mercuri \cite{cite_Mercuri1977}, and Lu \cite{cite_Lu2012}.  With these replacements, the compactness argument for
end-incompressible classes, and the Fuller-index invariance
argument go through without substantial change.

The circle
$S^1$ acts on $L_\beta X$ by reparametrization:
\[
        (s\cdot\gamma)(t)=\gamma(t+s).
\]
If some $\gamma\in L_\beta X$ had nontrivial stabilizer, then $\gamma$
would factor through a nontrivial covering $S^1\to S^1$,
contradicting indivisibility.
Therefore the $S^1$-action on $L_\beta X$ is free. Hence
\[
        H^{S^1}_*(L_\beta X,\mathbb Z)
        \cong
        H_*(L_\beta X/S^1,\mathbb Z),
\]
and we will construct a finite total rank model for the latter.

Let
\[
        \mathcal L_\beta X=W^{1,2}(S^1,X)_\beta
\]
be the Hilbert manifold of $H^1$-loops in the free homotopy class $\beta$.
Consider the energy functional
\[
        \mathcal E_g:\mathcal L_\beta X\to \mathbb R,
        \qquad
        \mathcal E_g(\gamma)=\int_{S^1}|\dot\gamma(t)|_g^2\,dt.
\]
Since $S(g,\beta)$ is compact, $\mathcal E_g$ is bounded above on the critical
set of $\mathcal E_g$ in class $\beta$.  Choose
\[
        E_0>\max_{\gamma\in S(g,\beta)}\mathcal E_g(\gamma)=M^2,
\]
which exists by the $\beta$-tautness assumption.  Set
\[
        \mathcal U_{E_0}
        =
        \{\gamma\in\mathcal L_\beta X\mid \mathcal E_g(\gamma)<E_0\}.
\]

\begin{lemma}\label{lem_inclusion}
The inclusion
\[
        \mathcal U_{E_0}\hookrightarrow \mathcal L_\beta X
\]
is an $S^1$-equivariant homotopy equivalence.  Consequently,
\[
        \mathcal U_{E_0}/S^1\simeq \mathcal L_\beta X/S^1.
\]
Moreover, the inclusion of smooth loops into the Hilbert loop space is an
$S^1$-equivariant homotopy equivalence, and hence
\[
        \mathcal L_\beta X/S^1\simeq L_\beta X/S^1.
\]
Therefore
\[
        \mathcal U_{E_0}/S^1\simeq L_\beta X/S^1.
\]
\end{lemma}

\begin{proof}
The energy functional $\mathcal E_g$ is smooth and $S^1$-invariant.  We
equip $\mathcal L_\beta X$ with the standard $S^1$-invariant Hilbert
metric induced by $g$.  With respect to this metric, the negative gradient
flow of $\mathcal E_g$ is $S^1$-equivariant.

We first recall the Palais--Smale compactness needed for the deformation
lemma.  Let $\{\gamma_j\}\subset \mathcal L_\beta X$ be a Palais--Smale
sequence for $\mathcal E_g$ with
\[
        \mathcal E_g(\gamma_j)\leq A.
\]
Then
\[
        \ell_g(\gamma_j)
        \leq
        \mathcal E_g(\gamma_j)^{1/2}
        \leq
        A^{1/2}.
\]
Since $\beta\in\pi_1^{\mathrm{inc}}(X)$,
Lemma~\ref{lemma_containedinCompact} implies that the images
$\gamma_j(S^1)$ lie in a fixed compact subset of $X$.  On this compact
set, the usual Hilbert-loop compactness argument for the geodesic energy gives
a strongly convergent subsequence in $\mathcal L_\beta X$.  Thus
$\mathcal E_g$ satisfies the Palais--Smale condition on sublevels.

By the choice of $E_0$, the functional $\mathcal E_g$ has no critical
points in
\[
        \{\mathcal E_g\geq E_0\}.
\]
Therefore, for every $A>E_0$, Palais' deformation lemma gives an
$S^1$-equivariant homotopy equivalence
\[
        \{\mathcal E_g<E_0\}\hookrightarrow \{\mathcal E_g<A\}.
\]
Letting $A\to\infty$, or equivalently applying the negative gradient flow
on the exhaustion by sublevels, we obtain an $S^1$-equivariant homotopy
equivalence
\[
        \mathcal U_{E_0}\simeq_{S^1}\mathcal L_\beta X.
\]
Passing to quotients gives
\[
        \mathcal U_{E_0}/S^1\simeq \mathcal L_\beta X/S^1.
\]

Finally, the inclusion
\[
        C^\infty(S^1,X)_\beta\hookrightarrow W^{1,2}(S^1,X)_\beta
\]
is an $S^1$-equivariant homotopy equivalence by standard equivariant
smoothing approximation for loops.  Since $L_\beta X$ denotes the same free
homotopy component in the smooth or compact-open loop space, we get
\[
        \mathcal L_\beta X\simeq_{S^1}L_\beta X.
\]
Combining the quotient equivalences gives
\[
        \mathcal U_{E_0}/S^1\simeq L_\beta X/S^1.
\]
\end{proof}

We now compare the Fuller index with the equivariant Morse model.  The only
subtlety is that after perturbing $g$ to a regular metric $g'$, the full
class $\beta$ geodesic set of $g'$ need not be compact.  We therefore work in
an isolating neighborhood of the compact lifted orbit block coming from $g$.

Let $g_S$ denote the Sasaki metric on the fixed contact manifold $C$ induced
by $g$ under the convention of Section~\ref{sec:Definition of F}.  Define the
$S^1$-invariant loop set $\widehat U_{\sqrt{E_0}}$ to consist of all loops
$\gamma:S^1\to C$ such that
\[
        [q\circ\gamma]=\beta,
        \qquad
        \ell_{g_S}(\gamma)<\sqrt{E_0}.
\]
Let
\[
        U_{\sqrt{E_0}}=\widehat U_{\sqrt{E_0}}/S^1
\]
be its image in the orbit-string quotient.  By \eqref{eq_sasaki} and the choice of $E_0$, this is an isolating
neighborhood of
\[
        N_0=\widetilde{\mathcal O}(g,\beta)
\]
in the sense of Lemma~\ref{lem_isolation}. It is also
$g _{S}$-controlled, in the sense of Definition
\ref{def_hcontrolled}, with the lower length bound supplied by the nontrivial
projected class $\beta$ and $K$ provided by Lemma
\ref{lemma_containedinCompact}. 

Let $W$ be as in Lemma~\ref{lem_isolation}, and let
$\epsilon$ be the corresponding Fuller constant for the isolating data
\[
        (N_0,U_{\sqrt{E_0}},W).
\]

Choose a $\beta$-regular complete metric $g'$ sufficiently
$C^2$-close $\epsilon'$-close to $g$
so that, on the fixed manifold $C$, the vector field $R^{\lambda_{g'}}$ is
$C^0\epsilon$-close to $R^{\lambda_g}$.  Choose $E'$ with
\[
        \max_{\gamma\in
				N_0}\ell_{g_S}(\gamma)<\sqrt{E'}<\sqrt{E_0}.
\]
Set $\widehat U$ to be the set of loops $\gamma:S^1\to C$ such that
\[
        [q\circ\gamma]=\beta,
        \qquad
        \ell_{g'_S}(\gamma)<\sqrt{E'},
\]
and set
\[
        U=\widehat U/S^1.
\]
After shrinking the perturbation if necessary, we may assume
\[
        W\subset U\subset U_{\sqrt{E_0}}.
\]
Thus Lemma~\ref{lem_continuation} gives
\[
        F(g,\beta)
        =
        i(N_1,R^{\lambda_{g'}}),
\]
where
\[
        N_1=U\cap\mathcal O(R^{\lambda_{g'}}).
\]

It remains to compute the last Fuller index.  Let
\[
        \mathcal U_{E'}^{g'}
        =
        \{\gamma\in\mathcal L_\beta X
        \mid \mathcal E_{g'}(\gamma)<E'\}.
\]
For a constant-speed $g'$-geodesic, its canonical lift has
$g'_S$-length equal to its $g'$-length.  Hence the critical circles of
$\mathcal E_{g'}$ in $\mathcal U_{E'}^{g'}$ are in one-to-one correspondence
with the elements of $N_1$.  Since $g'$ is $\beta$-regular, this critical set
is a finite union of nondegenerate critical circles
\[
        C_o\simeq S^1.
\]

Although the $S^1$-action on $\mathcal L_\beta X$ is only continuous,
each individual rotation is a smooth isometry preserving $\mathcal
E_{g'}$, so the negative gradient flow of $\mathcal E_{g'}$ commutes
with the action, and the deformations below are $S^1$-equivariant.
The Morse--Bott handlebody $\mathcal{MB}$ obtained from the unstable disk
bundles of these critical circles is $S^1$-invariant, and the standard
Hilbert-manifold Morse--Bott deformation theorem gives an $S^1$-equivariant
homotopy equivalence
\[
        \mathcal{MB}\simeq_{S^1}\mathcal U_{E'}^{g'}.
\]
After passing to the quotient by $S^1$, the unstable disk bundle over
$C_o$ contributes one cell of dimension $\operatorname{morse}(o)$;
here freeness of the action (indivisibility of $\beta$) is used, as
the quotient of the disk bundle by the free $S^1$-action is a single
cell. Therefore
\[
        \chi(\mathcal U_{E'}^{g'}/S^1)
        =
        \chi(\mathcal{MB}/S^1)
        =
        \sum_{o\in N_1}(-1)^{\operatorname{morse}(o)}.
\]

We now compare the $g'$-sublevel with the Hilbert sublevel for $g$.  Choose
numbers $a,b$  such that
\[
        M^2<a <E'<b<E_0
\]
and such that $\mathcal E_g$ has no critical points on
\[
        \{a\leq \mathcal E_g\leq b\}.
\]

For $\epsilon'$ small enough we have $\mathcal U^g_a \subset
\mathcal U^{g'}_{E'} \subset \mathcal U^g_b$, so that by
Lemma~\ref{lem_deformg'} the inclusion
\[
        \mathcal U_{E'}^{g'}\hookrightarrow \mathcal U_b^g
\]
is an $S^1$-equivariant homotopy equivalence.  Since all critical points of $\mathcal E_g$ have energy at most $M^2<b$, the inclusion
\[
        \mathcal U_b^g\hookrightarrow \mathcal U_{E_0}^g
\]
is also an $S^1$-equivariant homotopy equivalence.  Thus
\[
        \mathcal U_{E'}^{g'}\simeq_{S^1}\mathcal U_{E_0}^g.
\]
By Lemma~\ref{lem_inclusion},
\[
        \mathcal U_{E_0}^g/S^1\simeq L_\beta X/S^1.
\]
Consequently
\[
        \chi(\mathcal U_{E'}^{g'}/S^1)=\chi(L_\beta X/S^1).
\]
In particular $H_*^{S^1}(L_\beta X,\mathbb Z)$ has finite total rank.

On the other hand, since $g'$ is $\beta$-regular, the Fuller index formula,
summed over the finitely many lifted free homotopy classes met by $N_1$, gives
\[
        i(N_1,R^{\lambda_{g'}})
        =
        \sum_{o\in N_1}
        \frac{(-1)^{\operatorname{morse}(o)}}{\operatorname{mult}(o)}.
\]
Since $\beta$ is indivisible, all multiplicities in this class are $1$.  Hence
\[
        i(N_1,R^{\lambda_{g'}})
        =
        \chi(\mathcal U_{E'}^{g'}/S^1)
        =
        \chi(L_\beta X/S^1).
\]
Therefore
\[
        F(g,\beta)=\chi(L_\beta X/S^1)=\chi^{S^1}(L_\beta X).
\]
\qed

We state the following elementary lemma separately, since it may be of independent interest.

\begin{lemma}\label{lem_deformg'}
For a metric $h$ and a number $r$, set
\[
        \mathcal U_r^h
        =
        \{\gamma\in\mathcal L_\beta X\mid \mathcal E_h(\gamma)<r\}.
\]
Let $a<b$, and suppose that $\mathcal E_g$ satisfies the Palais--Smale
condition on the band
\[
        \mathcal B=\{a\leq \mathcal E_g\leq b\}
\]
and has no critical points there.  Let $g'$ be sufficiently $C^2$-close to
$g$, so that $\nabla\mathcal E_{g'}$ is $C^0$-close to
$\nabla\mathcal E_g$ on $\mathcal B$.  Suppose
\[
        \mathcal U_a^g\subset \mathcal C\subset \mathcal U_b^g,
\]
where
\[
        \mathcal C=\{\mathcal E_{g'}<c\}
\]
is an $S^1$-invariant sublevel set of $\mathcal E_{g'}$.  Then the
inclusion
\[
        \mathcal C\hookrightarrow \mathcal U_b^g
\]
is an $S^1$-equivariant homotopy equivalence.
\end{lemma}
\begin{proof}
By the Palais--Smale condition and the absence of critical points on
$\mathcal B$, there exists $\mu>0$ such that
\begin{equation}\label{eq_deltaEg}
        \|\nabla\mathcal E_g\|\geq \mu
\end{equation}
on $\mathcal B$.  For $g'$ sufficiently close to $g$, we have
\begin{equation}\label{eq_delta2Eg}
        \|\nabla\mathcal E_{g'}-\nabla\mathcal E_g\|<\mu/2
\end{equation}
on the same band.  Therefore
\[
        \langle \nabla\mathcal E_g,\nabla\mathcal E_{g'}\rangle>0
\]
on $\mathcal B$.

Since
\[
        \mathcal U_a^g\subset \mathcal C\subset \mathcal U_b^g,
\]
we have
\[
        \mathcal U_b^g-\mathcal C
        \subset
        \mathcal U_b^g-\mathcal U_a^g
        \subset
        \mathcal B.
\]
Now along a negative $\mathcal E_{g'}$-gradient trajectory 
lying in $\mathcal{B} $ we have: 
\[
        \frac{d}{dt}\mathcal E_g
        =
        -\langle \nabla\mathcal E_g,\nabla\mathcal E_{g'}\rangle
        <0.
\]
Note that the topological boundary of $\mathcal U^g_b - \mathcal C$
is contained in $\{\mathcal E_g = b\} \cup \{\mathcal E_{g'} =
c\}$. Consequently the negative $\mathcal E_{g'}$-gradient flow 
line of an element in $\mathcal U_b^g-\mathcal C $  
cannot exit $\mathcal U_b^g$ before hitting
$\mathcal C$.  Moreover, $\mathcal C$ is forward-invariant because it is a
sublevel set of $\mathcal E_{g'}$.

The estimates above also imply that $\mathcal E_{g'}$ has no critical points
on $\mathcal U_b^g-\mathcal C$.  The standard deformation argument of
Palais~\cite[Section 10, Proposition (2)]{cite_Palais1963}, applied to
$\mathcal E_{g'}$ on the region $\mathcal U_b^g-\mathcal C$, therefore
gives a deformation retraction of $\mathcal U_b^g$ onto $\mathcal C$.  The
energies and the Hilbert metrics are $S^1$-invariant, so the gradient flow is
$S^1$-equivariant.  Hence the deformation retraction is $S^1$-equivariant.
\end{proof}

\section{Product Formula} \label{sec_Product Formula}
In this section we prove a kind of product formula for the
$F$-count. This is used in all our applications. 
\begin{theorem}[Product formula] \label{thm:EulerProduct}
Let
\[
        Z\hookrightarrow (X,g) \xrightarrow{p} (Y, g _{Y})
\]
be a $\beta$-periodic fibration, with $Y$ closed,
connected and such that all smooth closed contractible $g
_{Y}$-geodesics in $Y$ are constant. Then $p$ is $\beta $-taut.
Fix $y_0\in Y$, set $Z_0=p^{-1}(y_0)$, and choose 
\[
        \beta_Z\in \pi_1^{\mathrm{inc}}(Z_0)
\]
such that
\[
        (i_{y_0})_*(\beta_Z)=\beta,
\]
where $i_{y_0}:Z_0\hookrightarrow X$ is the inclusion.
Denote by $g _{y _{0}}$ the restriction of $g$ to the fiber
of $y _{0}$.  Let
\[
        H=\pi_1(Y,y_0)
\]
act on $\pi_1^{\mathrm{inc}}(Z_{y_0})$ by holonomy, and set
\[
        S_{y_0}=H\cdot\beta_Z.
\]
Then   
\[
        F(g,\beta)
        =
        (-1) ^{\dim Y} \cdot \operatorname{card}(S_{y_0}) \cdot \chi(Y) \cdot
				F(g_{y_0},\beta_Z).
\] 
\end{theorem}
\begin{proof}
We prove $\beta $-tautness. We first show that every class $\beta$ $g$-geodesic in $X$ is vertical.
Let
\[
        o:S^1\to X
\]
be a closed $g$-geodesic representing the class $\beta$.  Since $\beta$
is a fiber class, the loop
\[
        p\circ o:S^1\to Y
\]
is contractible.  Since $p$ is a parallel Riemannian submersion, the local
de Rham product decomposition for $p$ implies that $p\circ o$ is a
$g_Y$-geodesic, possibly constant.  Hence $p\circ o$ is a smooth closed
contractible $g_Y$-geodesic.  By hypothesis, every such geodesic is
constant.  Therefore $p\circ o$ is constant, and so $o$ is contained in a
single fiber
\[
        Z_y=p^{-1}(y).
\]
Thus the class $\beta$ geodesic strings in $X$ are precisely the vertical
geodesic strings whose fiber class represents $\beta$ in $X$.

It remains to prove compactness.  Let
        $o_j$
be a sequence of class $\beta$ $g$-geodesic strings in $X$.  By the
previous paragraph, each $o_j$ is contained in a fiber $Z_{y_j}$.
Since $Y$ is compact, after passing to a subsequence we may assume
\[
        y_j\to y_\infty .
\]

Choose a contractible coordinate ball
\[
        U\subset Y
\]
around $y_\infty$.  For all large $j$, we have $y_j\in U$.  Since
$p$ is a parallel Riemannian submersion, the local de Rham product theorem
gives a local Riemannian product decomposition
\[
        p^{-1}(U)\cong Z_{y_\infty}\times U.
\]
Under this identification, vertical geodesics in the fibers $Z_{y_j}$ are
identified with geodesics of the fiber metric on $Z_{y_\infty}$, and their
fiber free homotopy classes lie in the finite holonomy orbit
\[
        S_{y_\infty}.
\]

By the fiber-tautness condition in the definition of a $\beta$-periodic
fibration, the union of the geodesic string sets
\[
        \bigcup_{\alpha\in S_{y_\infty}} S(g_{y_\infty},\alpha)
\]
is compact.  Since $S_{y_\infty}$ is finite, this is a finite union of
compact sets.  Therefore, after passing to a further subsequence, the vertical
geodesic strings $o_j$, transported to the model fiber $Z_{y_\infty}$,
converge to a fiber geodesic string.  Transporting back through the local
product chart gives convergence of the original $o_j$ in the class $\beta$
geodesic string space of $X$.

Thus every sequence of class $\beta$ geodesic strings has a convergent
subsequence.  Hence the class $\beta$ geodesic string set is compact, and
$g$ is $\beta$-taut.

Let
\[
        C=S^*X
\]
be the unit cotangent bundle of $(X,g)$, with projection
\[
        \operatorname{pr}:C\to X.
\]
For this fixed metric calculation, this is equivalent to the fixed-$C$
convention of Section~\ref{sec:Definition of F}.  Let $R^{\lambda_g}$ denote
the Reeb vector field on $C$.  The $F$-invariant is computed from the compact
lifted orbit set
\[
        N_0=\widetilde{\mathcal O}(g,\beta).
\]
The idea is to localize contributions over critical points of a Morse function
on $Y$.

Let
\[
        \pi^{\mathrm{vert}}:TX\to T^{\mathrm{vert}}X=\ker(dp)
\]
be the $g$-orthogonal projection.  By assumptions on $p:
X \to Y$, 
\[
        \nabla \pi^{\mathrm{vert}}=0,
\]
where $\pi^{\mathrm{vert}}$ is viewed as a $(1,1)$-tensor field.

For $\xi\in C$, let
\[
        v_\xi\in T_{\operatorname{pr}(\xi)}X
\]
be the $g$-dual unit vector to $\xi$.  Define
\[
        P:C\to \mathbb R
\]
by
\[
        P(\xi)=|\pi^{\mathrm{vert}}v_\xi|_g^2.
\]
Thus $P=1$ on vertical unit covectors and $P=0$ on horizontal unit
covectors.

Let $f:Y\to\mathbb R$ be a Morse function.  Define
\[
        \widetilde f:C\to\mathbb R
\]
by
\[
        \widetilde f=f\circ p\circ \operatorname{pr}.
\]
Equip $C=S^*X$ with the Sasaki metric $g_S$ induced by
$g$,  as in the Proof of Theorem \ref{theorem_Eulercharacteristic}. 
We shall use the following two elementary identities.

First,
\begin{equation}\label{eq_dpRlambda}
  dP(R^{\lambda_g})=0.
\end{equation}
%.
Indeed, let $\xi(\tau)$ be a flow line of $R^{\lambda_g}$, and set
\[
        x(\tau)=\operatorname{pr}(\xi(\tau)).
\]
Then $x(\tau)$ is a unit-speed $g$-geodesic.  Let
\[
        v(\tau)=\dot x(\tau)=v_{\xi(\tau)}
\]
be its velocity vector field.  Since $x(\tau)$ is a geodesic,
\[
        \nabla_\tau v=0.
\]
Since $\nabla \pi^{\mathrm{vert}}=0$, we have
\[
        \nabla_\tau(\pi^{\mathrm{vert}}v)
        =
        \pi^{\mathrm{vert}}(\nabla_\tau v)
        =
        0.
\]
Therefore
\[
\begin{aligned}
        \frac{d}{d\tau}P(\xi(\tau))
        &=
        \frac{d}{d\tau}|\pi^{\mathrm{vert}}v(\tau)|_g^2        \\
        &=
        2\left\langle
        \nabla_\tau(\pi^{\mathrm{vert}}v),
        \pi^{\mathrm{vert}}v
        \right\rangle_g                                      \\
        &=
        0.
\end{aligned}
\]
Hence $dP(R^{\lambda_g})=0$.

Second,
\begin{equation} \label{eq_dPNablaGS}
  dP(\nabla_{g_S}\widetilde f)=0.
\end{equation}
%.
Since $\widetilde f=f\circ p\circ \operatorname{pr}$ depends only on the
base point $\operatorname{pr}(\xi)$: 
\[
        \nabla_{g_S}\widetilde f\in T^{\mathrm{hor}}C.
\]
More precisely, $\nabla_{g_S}\widetilde f$ is the horizontal lift of
\[
        \nabla_g(f\circ p)
\]
on $X$.  Let $\xi(s)$ be a horizontal curve in $C$, and set
\[
        x(s)=\operatorname{pr}(\xi(s)).
\]
By definition of the Levi-Civita horizontal distribution on $S^*X$, the
covectors $\xi(s)$ are parallel along $x(s)$.  Equivalently, their
$g$-dual vectors
\[
        v(s)=v_{\xi(s)}
\]
are parallel along $x(s)$:
\[
        \nabla_s v(s)=0.
\]
Using again $\nabla \pi^{\mathrm{vert}}=0$, we get
\[
        \nabla_s(\pi^{\mathrm{vert}}v(s))
        =
        \pi^{\mathrm{vert}}(\nabla_s v(s))
        =
        0.
\]
Hence
\[
\begin{aligned}
        \frac{d}{ds}P(\xi(s))
        &=
        \frac{d}{ds}|\pi^{\mathrm{vert}}v(s)|_g^2        \\
        &=
        2\left\langle
        \nabla_s(\pi^{\mathrm{vert}}v(s)),
        \pi^{\mathrm{vert}}v(s)
        \right\rangle_g                                \\
        &=
        0.
\end{aligned}
\]
Taking $\dot\xi(0)=\nabla_{g_S}\widetilde f(\xi)$ gives
\[
        dP_\xi(\nabla_{g_S}\widetilde f)=0.
\]
Since $\xi$ was arbitrary,
\[
        dP(\nabla_{g_S}\widetilde f)=0.
\]

Now set
\[
        V_\varepsilon
        =
        R^{\lambda_g}
        -
        \varepsilon\nabla_{g_S}(P+\widetilde f).
\]

Recall first the discussion of Fuller continuation in the Appendix.  The set
$N_0=\widetilde{\mathcal O}(g,\beta)$ is compact by $\beta$-tautness.  Choose a controlled isolating neighborhood $U$ of $N_0$ and an open
neighborhood $W$ with $\overline W\subset U$, and let $\epsilon'$ be the
corresponding Fuller constant.  Take
$\varepsilon>0$ small enough so that $V_\varepsilon$ is $C^0$ $\epsilon'$-close
to $R^{\lambda_g}$ on the compact region determined by $U$.  Lemma~\ref{lem_continuation} gives a compact continued set $N_\varepsilon$ and
\[
        F(g,\beta)
        =
        i(N_\varepsilon,V_\varepsilon).
\]
Note that
the above does not require any further properties of $V_\varepsilon$, such as
completeness.  We set up the Fuller index theory to sidestep such issues.

We now identify the closed orbits in $N_\varepsilon$.  We
show that $P$ is a Lyapunov function for $V _{\epsilon}$.
Mainly, along a trajectory $\xi(\tau)$ of $V_\varepsilon$, we have
\[
\begin{aligned}
        \frac{d}{d\tau}P(\xi(\tau))
        &=
        dP(R^{\lambda_g})
        -
        \varepsilon dP(\nabla_{g_S}P)
        -
        \varepsilon dP(\nabla_{g_S}\widetilde f)          \\
        &=
        -\varepsilon|\nabla_{g_S}P|_{g_S}^2
        \leq 0.
\end{aligned}
\]
Thus $P$ is nonincreasing along the $V_\varepsilon$-flow.

If $\xi(\tau)$ is a closed orbit of $V_\varepsilon$, then
$P(\xi(\tau))$ returns to its initial value after one period.  Since $P$
is nonincreasing, it must be constant along the orbit.  Hence
\[
        \nabla_{g_S}P=0
\]
along the orbit.

On each unit cotangent sphere, $P$ is the squared norm of the vertical
component of the corresponding unit vector.  Therefore the critical locus of
$P$ consists of the two critical loci
\[
        P=1
        \qquad\text{and}\qquad
        P=0,
\]
corresponding respectively to vertical and horizontal covectors.  At
$\varepsilon=0$, the lifted orbit set under consideration
is vertical, hence lies in $P=1$.  Taking the controlled
neighborhood $U$ (above) so that all its elements are
disjoint (as loops) from $P=0$, we get that every closed orbit in $N_\varepsilon$ lies in the vertical critical locus  $P=1$.
Consequently, every closed orbit in $N_\varepsilon$ projects to a geodesic
contained in a fiber $Z_y$.

It remains to see over which fibers these vertical orbits occur.  On $P=1$, the Reeb vector field restricts to the geodesic Reeb
field of the fiber $(Z_y,g_y)$, while $\widetilde f$ depends only on the
base point $y$.  The perturbation
\[
        -\varepsilon\nabla_{g_S}\widetilde f
\]
moves the base point in the negative gradient direction of $f$.  Therefore
a vertical closed orbit of $V_\varepsilon$ can occur only over a critical
point of $f$.  Conversely, if $y\in \operatorname{Crit}(f)$, then
\[
        \nabla_{g_S}\widetilde f=0
\]
over $Z_y$, and the vertical part of $V_\varepsilon$ is just the fiber
geodesic flow, up to the harmless $P$-term which vanishes on the vertical
critical locus.  Thus the closed orbits of $V_\varepsilon$ in
$N_\varepsilon$ are precisely the vertical closed geodesic orbits lying in
fibers over critical points of $f$.

We now compute the index.  For $y\in Y$, put
\[
        H_y=\pi_1(Y,y),
\]
acting on $\pi_1^{\mathrm{inc}}(Z_y)$ by holonomy.  Choose
\[
        \beta_Z\in \pi_1^{\mathrm{inc}}(Z_y)
\]
with
\[
        (i_y)_*(\beta_Z)=\beta.
\]
Set
\[
        S_y=H_y\cdot \beta_Z.
\]

The Fuller index is additive over disjoint compact open subsets.  Thus it remains to compute the local index contribution
$i(N_{\varepsilon,y},V_\varepsilon)$ corresponding to the isolated subset of
$N_\varepsilon$ lying over a critical point $y$ of $f$.  We defer this
computation to the Appendix as Lemma~\ref{lem_fuller-local-product-geodesic};
it yields that
\begin{equation}\label{eq_localindex}
        i(N_{\varepsilon,y},V_\varepsilon)
        =
        (-1)^{\dim Y+\operatorname{morse}(y)}
        \sum_{\alpha\in S_y}F(g_y,\alpha),
\end{equation}

Next, we show that $F(g_y,\cdot)$ is constant on $S_y$.  Indeed, let
\[
        \gamma: [0,1] \to Y
\]
be a path based at $y$.  Pulling back $p:X\to Y$ along $\gamma$ gives a
fibration over $[0,1]$.  By the same proof that $p$ is
$\beta $-taut, this
pullback is a taut deformation of the fiber metric from $(Z_y,g_y)$ back to
itself, with final fiber identified with the initial one by the holonomy map
$h_\gamma:Z_y\to Z_y$.  Fuller index invariance under taut deformation gives
\[
        F(g_y,\alpha)=F(g_y,h_{\gamma,*}\alpha).
\]
Thus $F(g_y,\alpha)$ is constant on the holonomy orbit $S_y$.  Hence
\[
        \sum_{\alpha\in S_y}F(g_y,\alpha)
        =
        \#S_y \cdot F(g_y,\beta_Z).
\]
If $y,y' \in Y$, holonomy along any path from $y$ to $y'$ identifies
$S_y$ with $S_{y'}$.  The same path gives a taut deformation between the fiber metrics, so
\[
        \sum_{\alpha\in S_{y}} F(g_{y},\alpha)
\]
is independent of $y$.  
So we get:
\begin{align*}
  F (g, \beta ) & = \sum_{y \in \operatorname {Crit} (f)}
	i(N_{\varepsilon,y}, V_\varepsilon) \\
        & =
        \sum_{y \in \operatorname {Crit} (f)}
(-1)^{\dim Y+\operatorname{morse}(y)}
        \sum_{\alpha\in S_y}F(g_y,\alpha), \\
				& = \sum_{y \in \operatorname {Crit} (f)} (-1)^{\dim
				Y+\operatorname{morse}(y)} \cdot \# S _{y _{0}}
				\cdot F (g
				_{y _{0}}, \beta _{Z}) \\
				& = (-1)^{\dim Y} \cdot \chi(Y) \cdot \# S _{y _{0}}
				\cdot F (g
				_{y _{0}}, \beta _{Z}).
\end{align*}
%. 
% We denote this common fiber contribution by
% \[
%         A.
% \]

\end{proof}
\begin{corollary}  \label{cor_productformula}
Let $X = Z \times Y$, with $Y$ closed,
connected, such that all smooth closed contractible $g
_{Y}$-geodesics in $Y$ are constant and such that $Z,
g _{Z}$ is taut. Let $\beta \in \pi _{1} ^{inc} (X) $ be
a fiber class. Then $g= g _{Z} \times g _{Y}$  is $\beta $-taut.
Furthermore,   
\[
        F(g,\beta)
        =
        (-1) ^{\dim Y} \cdot \chi(Y) \cdot
				F(g_{Z},\beta_Z).
\] 
\end{corollary}
\begin{proof} [Proof] This is immediate from the theorem. \end{proof}
\section{Proof of Theorem \ref{thm:Fibration} and
its corollaries}
\label{sec:Proof of Theorem fibration}
We prove the first part. By Cartan--Hadamard, a complete negatively curved
Riemannian manifold has at most one closed geodesic in each nontrivial free
homotopy class; moreover the absence of nontrivial periodic Jacobi fields in
negative curvature gives nondegeneracy. Existence of a geodesic in the end-incompressible class $\beta$ 
follows from Lemma~\ref{lem_existsgeod}.  Hence a complete negatively curved
metric would give a unique nondegenerate class $\beta$
geodesic string. It thus suffices to prove the second part.

We prove the second part. Suppose first that $\beta
$ is atomic, so that $\beta $ has a representative that is
a $k$-cover of a loop, $k \geq 1$,  whose corresponding class $\alpha $ is indivisible.   By the assumption that
all contractible $g _{Y}$ geodesics are constant, classical 
Morse theory Milnor~\cite{cite_MilnorMorsetheory} tells us
that $Y$ has vanishing higher homotopy groups $\pi _{j} (Y,
y _{0})$, $j \geq 2$. Thus by the long exact sequence of a fibration $i _{Z, *}: \pi
_{1} (Z, z _{0}) \to \pi _{1} (X, i_Z (z _{0})) $ is a group
injection, and $\image i _{Z,* } = \ker p _{*}$.
Now $\pi _{1} (Y, y _{0}) $ is torsion-free since it is
a finite dimensional $K (\pi _{1} (Y, y _{0}),1  )$ space, 
Hatcher ~\cite[Prop. 2.45]{cite_HatcherAlgebraic}. 
It follows that $\alpha$ is also a fiber class since $$p _{*}
(\alpha ^{k}) = 0 \implies p_* (\alpha ) =0.$$  Furthermore,
$\alpha \in \pi _{1} ^{inc} (X) $ (otherwise $\beta
$ isn't).

\begin{lemma}\label{lem_powerperiodic}
Let
\[
        Z\hookrightarrow (X,g)\xrightarrow{p}(Y,g_Y)
\]
be a $g$-parallel Riemannian submersion with complete total space, and assume
that $Y$ is connected with $\pi_2(Y)=0$.  Let
\[
        \beta,\gamma\in \pi_1^{\mathrm{inc}}(X)
\]
be fiber classes such that
\[
        \beta=\gamma^m
\]
for some $m\geq1$.  If $g$ is $\beta$-taut, then $g$ is
$\gamma$-taut and $p$ is $\gamma$-holonomy-finite.
\end{lemma}
\begin{proof}
We show $\gamma$-tautness.
Every closed $\gamma$-geodesic $o_\gamma$ has $m$-fold cover $o_\gamma^{\,m}\in S(g,\beta)$, so
\[
  \ell_g(o_\gamma)=\tfrac1m\,\ell_g(o_\gamma^{\,m})\le A/m,
  \qquad A:=\max_{o\in S(g,\beta)}\ell_g(o)<\infty,
\]
where finiteness of $A$ is $\beta$-tautness. Hence
$S(g,\gamma)=S_{A/m}(g,\gamma)$, which is compact by
Lemma~\ref{lem_SA}. 

We show $\gamma$-holonomy-finiteness. Choose
\[
        \gamma_Z\in\mathcal R(\gamma)
\]
and set
\[
        S_{y_0}=\pi_1(Y,y_0)\cdot\gamma_Z.
\]
Suppose, toward a contradiction, that $S_{y_0}$ is infinite.  For each
\[
        \gamma_i\in S_{y_0}
\]
choose a closed $g_{y_0}$-geodesic
\[
        o_i\subset Z_{y_0}
\]
representing $\gamma_i$.  Such geodesics exist by
Lemma~\ref{lem_existsgeod}.  Since the classes $\gamma_i$ are distinct in
$Z_{y_0}$, the loops $o_i$ are pairwise distinct as elements of the fiber
loop space.

We first record a separation property for these loops in the ambient loop
space $LX$.  Choose a contractible trivializing neighborhood
\[
        U\subset Y
\]
of $y_0$.  For a loop $o\subset Z_{y_0}$, let
\[
        \mathcal V_{U,o}\subset LX
\]
be the set of loops $\eta:S^1\to X$ such that
\[
        p(\eta(S^1))\subset U
\]
and such that $\eta$, regarded as a loop in $p^{-1}(U)$, is freely
homotopic in $p^{-1}(U)$ to $o$.  The set $\mathcal V_{U,o}$ is open in
$LX$: the condition that the image lie in $p^{-1}(U)$ is open, and free
homotopy classes are open components of the loop space of $p^{-1}(U)$.

Moreover, if $o_i$ and $o_j$ represent distinct free homotopy classes in
$Z_{y_0}$, then
\[
        \mathcal V_{U,o_i}\cap \mathcal V_{U,o_j}=\varnothing.
\]
Indeed, an element of the intersection would give a free homotopy between
$o_i$ and $o_j$ inside $p^{-1}(U)$.  Since $U$ is contractible, the
inclusion
\[
        Z_{y_0}\hookrightarrow p^{-1}(U)
\]
is a homotopy equivalence, and hence $o_i$ and $o_j$ would already be
freely homotopic inside $Z_{y_0}$, contrary to construction.

On the other hand, $S(g,\gamma)$ is compact.
Thus the infinite family $\{o_i\}\subset S(g,\gamma)$ has an accumulation
point $o_\infty\in S(g,\gamma)$.  Since every $o_i$ lies in the fiber
$Z_{y_0}$, the limit $o_\infty$ also lies in $Z_{y_0}$.  The open
neighborhood $\mathcal V_{U,o_\infty}$ of $o_\infty$ therefore contains
$o_i$ for infinitely many $i$, forcing these $o_i$'s to lie in the same
fiber free homotopy class as $o_\infty$.  This contradicts the fact that the
classes $\gamma_i$ are pairwise distinct.  Hence $S_{y_0}$ is finite and
$p$ is $\gamma$-holonomy-finite.

\end{proof}

By Theorem \ref{thm:EulerProduct} $p$ is $\beta $-taut. By
the lemma above $p$ is $\alpha $-holonomy finite
and hence is an $\alpha $-periodic fibration.
Now, if $\chi(Y) \neq \pm 1$ then by the product formula of
Theorem \ref{thm:EulerProduct} $F (g, \alpha
) \neq  \pm 1$, since $F (g _{y _{0}}, \alpha _{Z})$ is an integer by
Theorem \ref{theorem_Eulercharacteristic}.  
% By Theorem \ref{theorem_Eulercharacteristic}
% \begin{equation*}
% F (g, \alpha ) = \chi ^{S ^{1}}(L	_{\alpha}X).
% \end{equation*}

Now, if $X$ admits a complete metric $h$ with a unique and
nondegenerate class $\beta$ $h$-geodesic string $o _{\beta }$, then it
clearly also has a unique class $\alpha$ $h$-geodesic
string $o _{\alpha }$. Furthermore, $o _{\alpha }$ must be
nondegenerate as otherwise $o _{\beta }$, which is
a $k$-fold cover of $o _{\alpha }$ is degenerate.
Then we have: $$F (h, \alpha) = \pm 1.$$ 
But this is a contradiction, since $F (h, \alpha) = F (g,
\alpha )$ by Theorem \ref{theorem_Eulercharacteristic}.

We now prove the general case. Suppose by contradiction that $X$
admits a complete metric $h$ with a unique and nondegenerate class $\beta$
$h$-geodesic string $o _{\beta}$. By the preceding paragraph, the class $\beta$ is not atomic.    
Now $o _{\beta}$ covers a multiplicity one geodesic
string $o _{\alpha}  $ in some class $\alpha \in  \pi _{1}
^{inc} (X)$. Moreover, $o _{\alpha }$ is the unique
geodesic string in its class, otherwise $o _{\beta}$ would not be unique in its class.
We prove that $\alpha $ is indivisible, which will be a contradiction to $\beta $ not being atomic and will complete the proof.

Suppose otherwise, so that  $\alpha  _{x
_{0}}  = \rho  ^{k}$ for $k>1$ and $\rho \in \pi _{1} (X,
x _{0})$, (with the notation of Definition
\ref{definition_indivisible}). 
Let $o _{\rho}$ be a class $\rho $, closed
$h$-geodesic string (where by slight abuse $\rho$ also denotes the class
in $\pi _{1}  ^{inc} (X)$ corresponding to the based class
$\rho $.) 
It is immediate
that the $k$ cover of $o _{\rho }$, $o _{\rho } ^{k}$ represents
$\alpha$ and is a $h$-geodesic string. By the
uniqueness, ${o} _{\alpha} = o _{\rho } ^{k}$. But this contradicts
multiplicity one of ${o} _{\alpha}$. So $\alpha $ is
indivisible.

\qed
\begin{proof} [Proof of Corollary \ref{corollary_product}]
Let $g _{Z}, g _{Y}$ be complete
Riemannian metrics with $g _{Z}$ taut, and $g _{Y}$ not admitting contractible
closed geodesics. Take the product metric $g = g _{Z} \times
g _{Y}$ on $X = Z \times Y$.  For a class $\beta \in \pi
^{inc}_{1} (X)$ in the image of the inclusion $\pi _{1}
^{inc} (Z) \to \pi _{1} ^{inc} (X)$, the natural projection
$X \to Y$ is automatically a $\beta $-periodic fibration. Then the
conclusion readily follows from Theorem \ref{thm:Fibration}.
\end{proof}
\begin{proof} [Proof of Corollary \ref{corollary_Thurston0}]
The first two parts follow from Theorem
\ref{thm:Fibration}, Lemma
\ref{lem:isometry-finite-orbits-finite-type} and Example
\ref{example_phi}.
The last part is deduced as follows. The product formula
Theorem \ref{thm:EulerProduct} gives us $F (g _{\phi },
\beta ) = 0$.  So for any $\beta $-regular $g$ with finitely
many class $\beta$ $g$-geodesic strings $o$, since $\beta
$ is indivisible, by Corollary \ref{cor_chi}  we have:
$$\chi ^{S ^{1}}(L _{\beta} X _{\phi } ) = 
F (g, \beta ) = \sum _{o} (-1) ^{\operatorname {morse}  o} =0.$$ Then the
conclusion readily follows.
\end{proof}

\begin{proof} [Proof of Theorem \ref{theorem_valuesOfInvariant}]
By Theorem \ref{thm:EulerProduct} 0 is certainly a value of
the invariant $F$. We first prove that every negative rational number is
the value of the invariant. 
Let $p,q $ be positive integers. Let $Y
$ be a closed surface of genus $(p +1) >1$ with a hyperbolic
metric $g _{Y}$, let $Z$ be the
genus 2 closed surface with a hyperbolic metric $g _{Z}$ and
let $\beta _{Z} \in \pi _{1} ^{inc} (Z)$ be
the class represented by a $2 \cdot q$-fold covering of
a simple closed loop representing a generator of the fundamental group of $Z$.

Let $X = Y \times Z$ with the product metric $g = g _{Y}
\times g _{Z}$ and $p: X \to Y $  the
canonical projection. Taking $\beta = (i _{Z}) _{*} (\beta
_{Z})$, by Corollary \ref{cor_productformula}
since $(-1) ^{\dim Y} = 1$,  $$F
(g, \beta) = \chi(Y)  \cdot F (g _{Z}, \beta _{Z}) = (
- 2p) \cdot \frac{1}{2q} = - \frac{p}{q}. $$
So we
proved our first claim.

Let again $p,q$ be positive integers. Let $Y$ be closed surface of
genus $2$, with a hyperbolic metric $g _{Y}$. And let $Z$ be a manifold satisfying $F (g _{Z}, \beta
_{Z}) = -\frac{p}{2q}$ for some $\beta _{Z}$-taut metric $g
_{Z}$ on $Z$ and for some class $\beta _{Z} \in \pi _{1} ^{inc} (Z)$. This
exists by the discussion above. Let $g = g _{Y} \times
g _{Z}$ be the product metric on $Y \times Z$, and $\beta
 $ as above. Analogously to the discussion above we get:
$$F (g, \beta) = \chi(Y)  \cdot F (g _{Z}, \beta _{Z}) = (
-2) \cdot \frac{-p}{2q} =  \frac{p}{q}. $$
\end{proof}

% ------------------------------------------------------------------
% Near-flat rigidity via KAM confinement (Proposition A).
% Self-contained section for GromovFuller.tex.
%
% External references into the main paper:
%   Theorem \ref{thm:notgeodesible}   (bounded length spread => taut family)
%   Lemma   \ref{lem_SA}              (compactness of S_A(g,\beta))
% Bibliography: append nearflat-rigidity.bib to link.bib.
% ------------------------------------------------------------------

\section{Near-flat rigidity via KAM confinement}
\label{sec_nearflat}

We show that every Riemann-Finsler metric on $T^2$ sufficiently close
to a flat metric $g_0$ is taut, with every class $\beta$ closed geodesic
of length at most $(1+\delta)$ times the minimal $g_0$-length in the
class, and is taut homotopic to $g_0$.

\begin{proposition}[Near-flat rigidity]
\label{prop_nearflat}
Let $g_0$ be a flat Riemannian metric on $T^2$. There is an absolute $k_0 \in
\mathbb{N}$ such that for every $\delta
\in (0,1]$ there is $\varepsilon > 0$ with the following property. If
$g'$ is a Riemann-Finsler metric on $T^2$ with $\|g' - g_0\|_{C^{k_0}}
< \varepsilon$, then for \emph{every} nontrivial free homotopy class
$\beta$, every closed $g'$-geodesic $o$ of class $\beta$ satisfies
\begin{equation}
\label{eq_lengthbound}
\ell_{g'}(o) \;\le\; (1+\delta)\, L_\beta,
\end{equation}
where $L_\beta$ denotes the minimal $g_0$-length in class $\beta$.
Consequently $g'$ is taut, and any two metrics in the $\varepsilon$-ball
are joined by a taut deformation; in particular every such $g'$ is taut
homotopic to $g_0$.
\end{proposition}

Note the order of quantifiers: $\varepsilon$ does not depend on $\beta$.
For Finsler $g'$, $\|g' - g_0\|_{C^{k_0}}$ denotes the $C^{k_0}$-distance
of $F'^2$ to $F_0^2$ on $\{\tfrac12 \le F_0 \le 2\} \subset TT^2$.

The classical theory of minimal geodesics on the torus (Morse
\cite{cite_Morse1924}, Hedlund \cite{cite_Hedlund1932}, Bangert
\cite{cite_Bangert1990MinimalGeodesics}, Mather
\cite{cite_Mather1990}) controls, for every metric, the
\emph{minimizers} in each class. Proposition \ref{prop_nearflat} is
orthogonal to it: the bound is on \emph{all} class $\beta$ closed
geodesics, not just the minimizers, and no minimizer-theoretic
argument reaches it. The mechanism is confinement by invariant tori
of the geodesic flow which are graphs over $T^2$ and hence, in pairs,
separate the unit tangent bundle. Intrinsically this graph property is the second Birkhoff theorem,
established for Lagrangian invariant tori by Bialy--Polterovich and
Polterovich \cite{cite_BialyPolterovich1989LagrangianTori,
cite_Polterovich1991SecondBirkhoff,
cite_BialyPolterovich1992BirkhoffTheorem,
cite_BialyPolterovich1992LagrangianDistributions} and, without a
dynamical assumption, by Arnaud \cite{cite_Arnaud2010Birkhoff}; for
the merely continuous invariant tori of $T^2$ geodesic flows it is due
to Bialy \cite{cite_Bialy1989AubryMather} (see also
\cite{cite_BialyPolterovich1986}). In the perturbative regime below
it comes for free from KAM. E.\ Hopf's rigidity theorem
\cite{cite_Hopf1948}---$ST^2$ is foliated by incompressible invariant
tori if and only if the metric is flat---explains why the smallness of
$\varepsilon$ cannot be removed: total confinement characterizes
flatness, and away from $g_0$ one has only the Cantor family of
surviving tori that KAM provides, which suffices for
\eqref{eq_lengthbound}.
\subsection{Setup and normalization}
\label{subsec_nearflat_setup}
Every flat torus is isometric to $\mathbb{R}^2/\Lambda$ with the
Euclidean metric; composing with a linear identification
$\mathbb{R}^2/\mathbb{Z}^2 \cong \mathbb{R}^2/\Lambda$ produces a
diffeomorphism $\phi$ of $T^2$ with $\phi^*g_0$ constant-coefficient.
Pulling the problem back by $\phi$ distorts $C^{k_0}$-balls by a fixed
constant and preserves geodesics, lengths and classes. So assume $T^2 = \mathbb{R}^2/\mathbb{Z}^2$ and $g_0 = G_0$, a constant
positive-definite matrix; to simplify notation further we take $G_0$
to be the identity, so that $g_0$ is the standard flat metric. The general case is identical, with the constants below depending on
$G_0$; see also Remark \ref{rem_nearflat_scope}. Here and below $C$
denotes a positive absolute constant whose value may change between
occurrences, while each numbered constant $C_1, C_2, C_3$ denotes a
single specific absolute constant, produced at its first appearance
and unchanged thereafter. A
nontrivial free homotopy class $\beta$ simultaneously represents a
deck vector $\beta \in \mathbb{Z}^2 - \{0\}$, and $L_\beta = |\beta
|$, the Euclidean norm.

The unit-speed geodesic flow of $g'$ is the Hamiltonian flow of
$H'(x,p) = \tfrac12 F'^*(x,p)^2$ on $\Sigma' = \{H' = \tfrac12\} \subset
T^*T^2$, where $F'^*$ is the dual norm; a periodic orbit of period $T$
projects to a closed geodesic of length $T$. The fiberwise Legendre dual
costs at most one derivative, so
\begin{equation*}
\|H' - H\|_{C^{k_0-1}(A)} \le C\varepsilon
\quad \text{on} \quad A = \{\tfrac14 \le H \le 1\},
\qquad H(p) = \tfrac12 |p|^2;
\end{equation*}
we rename $k_0 - 1 \to k_0$ and, enlarging $\varepsilon$ by this
fixed factor, assume $\|H' - H\|_{C^{k_0}(A)} \le \varepsilon$.

The flat system is integrable: $\dot x = p$, $\dot p = 0$. Its unit
level $\Sigma = \{H = \tfrac12\}$ is
foliated by the invariant tori $T_\psi = \{p = (\cos \psi, \sin
\psi)\}$, $\psi \in S^1$; on $T_\psi$ the flow is linear with
frequency vector $(\cos \psi, \sin \psi)$. Here and for the rest of
the section $S^1 := \mathbb{R}/2\pi\mathbb{Z}$ denotes the angle
circle, of total measure $2\pi$---in contrast with the parametrizing
circle $\mathbb{R}/\mathbb{Z}$ of loops used elsewhere in the paper. For $\varepsilon$ small, $D^2_p H'$ is close to the
identity, so each locus $\{p : H'(x,p) = \tfrac12\}$ is a strictly
convex curve enclosing the origin; it is therefore the radial graph of
a function $\rho(x, \cdot) : S^1 \to \mathbb{R}_{>0}$ with $\|\rho -
1\|_{C^1} \le C\varepsilon$. In particular,
\begin{equation*}
\Psi : \Sigma' \to T^2 \times S^1, \qquad (x,p) \mapsto (x, \psi := \arg
p)
\end{equation*}
is a diffeomorphism. For $(x, p) \in \Sigma'$, with $\psi = \arg p$
as above, the radial graph description of the fiber says exactly that
$p = \rho(x, \psi)\, (\cos \psi, \sin \psi)$; hence
\begin{equation*}
|p - (\cos \psi, \sin \psi)| = |\rho(x, \psi) - 1| \le
C\varepsilon.
\end{equation*}

\subsection{The KAM input}
\label{subsec_nearflat_KAM}
For a real number $r > 0$, call $\omega \in \mathbb{R}^2$
\emph{$(r,2)$-Diophantine} if
$|\langle \gamma, \omega \rangle| \ge r |\gamma|^{-2}$ for every integer
vector $\gamma \in \mathbb{Z}^2 - \{0\}$, and call $\psi \in S^1$
\emph{good} if
$(\cos \psi, \sin \psi)$ is $(r,2)$-Diophantine. The notion of good
depends on the parameter $r$, which the notation suppresses; $r$ is an
adjustable parameter, chosen small in the proof of Proposition
\ref{prop_nearflat}.

\begin{theorem}[Isoenergetic KAM]
\label{thm_isoKAM}
There is an absolute $k_0$ such that for every $r, \eta > 0$ there
is $\varepsilon_0 = \varepsilon_0(r, \eta) > 0$ with the
following property. If $\|H' - H\|_{C^{k_0}(A)} \le \varepsilon_0$,
then for every good $\psi_0$ there is a $2$-torus $T'_{\psi_0} \subset
\Sigma'$, invariant under the $H'$-flow, of the form $T'_{\psi_0} =
\{(x, P(x)) : x \in T^2\}$ with $\sup_x |P(x) - (\cos \psi_0, \sin \psi_0)| \le \eta$.
\end{theorem}

This is the isoenergetic KAM theorem, in the graph form in which the
surviving tori are labelled by their frequency rays; see Arnold
\cite{cite_Arnold1963} for the analytic case, Broer--Huitema
\cite{cite_BroerHuitema1991} for the derivation of the isoenergetic
statement from the ordinary one, and P\"oschel \cite{cite_Poschel1982,
cite_Poschel2001} for quantitative versions in finite smoothness. The
one hypothesis to verify is isoenergetic nondegeneracy of $H$, i.e.\
the nonvanishing of the bordered determinant, which here is immediate:
\begin{equation*}
\det
\begin{pmatrix}
D^2 H & \nabla H \\
\nabla H^\top & 0
\end{pmatrix}
= -|p|^2 = -1 \quad \text{on } \{H = \tfrac12\}.
\end{equation*}
Alternatively---and this is the
source of the absolute constant $k_0$---on each of four sectors of
$S^1$ on which a coordinate velocity is bounded away from zero, the
flow on $\Sigma'$ reduces to an exact area-preserving twist map of an
annulus whose twist is, up to a positive factor, the bordered
determinant; Moser's invariant curve theorem \cite{cite_Moser1962}, in
Herman's optimal-smoothness form \cite{cite_Herman1983}, then produces
the invariant tori, and $C^1$-closeness to the flat sections makes them
automatically graphs.

\begin{lemma}[Good directions are dense at scale $r$]
\label{lem_gooddense}
There is an absolute constant $C_1$ such that, for every $r > 0$,
$\operatorname{Leb}\big(S^1 - \{\text{good}\}\big) \le C_1
r$. Since the good set is closed, every complementary gap has
length $\le C_1 r$. In particular, when $C_1 r < 2\pi$ the good set
is nonempty and every interval in $S^1$ of length $> C_1 r$
contains a good angle.
\end{lemma}

\begin{proof}
The good set is the complement of the union of the sets
\begin{equation*}
B_{\gamma} = \{\psi \in S^1 : |\langle \gamma, (\cos \psi, \sin
\psi) \rangle| < r
|\gamma|^{-2}\},
\end{equation*}
over integer vectors $\gamma \in \mathbb{Z}^2 - \{0\}$. Each
$B_{\gamma}$ is open, so the good set is closed. Writing $\gamma = |\gamma|
(\cos \theta_{\gamma}, \sin \theta_{\gamma})$, we have $\langle \gamma, (\cos \psi, \sin \psi) \rangle = |\gamma|
\cos(\psi - \theta_{\gamma})$, so that
\begin{equation*}
B_{\gamma} = \{\psi \in S^1 : |\cos(\psi - \theta_{\gamma})| < r
|\gamma|^{-3}\}.
\end{equation*}
For $0 \le s \le 1$ the set $\{\phi \in S^1 : |\cos \phi| < s\}$
is the union of two arcs of length $2 \arcsin(s)$, centered at $\pm
\pi/2$, and $\arcsin(s) \le \tfrac{\pi}{2}\, s$ by convexity of
$\arcsin$ on $[0,1]$; hence
\begin{equation*}
\operatorname{Leb}(B_{\gamma}) \le 2\pi\, r\, |\gamma|^{-3},
\end{equation*}
trivially also when $r |\gamma|^{-3} > 1$. The number of integer vectors
$\gamma$ with $n \le |\gamma| < n+1$ is at most $Cn$, so
\begin{equation*}
\operatorname{Leb}\big(S^1 - \{\text{good}\}\big) \;\le\;
\sum_{\gamma \neq 0} \operatorname{Leb}(B_{\gamma}) \;\le\; 2\pi\,
r \sum_{n \ge 1} \frac{Cn}{n^{3}} \;\le\; C_1 r.
\end{equation*}
Finally, a complementary gap is a connected component of the open set
$S^1 - \{\text{good}\}$, so its length is at most the measure $C_1 r$
of that set. When $C_1 r < 2\pi$ the complement cannot be all of
$S^1$, so the good set is nonempty, and an interval of length $> C_1
r$ cannot lie in a single gap, so it contains a good angle.
\end{proof}

\begin{lemma}[Confinement]
\label{lem_confinement}
For $r$ and $\eta$ sufficiently small and $\varepsilon \le
\varepsilon_0(r, \eta)$, along every orbit of
the $H'$-flow on $\Sigma'$ the momentum angle $\psi(t)$ stays in a fixed
arc of length $W := C_2(r + \eta)$, where one may take $C_2 = 4 C_1
+ 24$.
\end{lemma}

\begin{proof}
In the coordinates $\Psi$, each $T'_{\psi_0}$ is the graph of
$\Theta_{\psi_0}(x) = \arg P(x)$, and for $\eta \le 1/4$ we have
$|\Theta_{\psi_0}(x) - \psi_0| \le 4\eta$: indeed $|P(x)| \ge 1 -
\eta$, so the sine of the angle between $P(x)$ and $(\cos \psi_0,
\sin \psi_0)$ is at most $\eta/(1-\eta)$, and hence the angle is at
most $4\eta$, using $\theta \le \tfrac{\pi}{2} \sin \theta$ for
$\theta \in [0, \pi/2]$. If the orbit starts on a torus it stays
there and $\psi$ oscillates by at most $8\eta$. Otherwise, the two intervals below are arcs of length $2 C_1 r$,
disjoint since $r$ and $\eta$ are small, so by Lemma
\ref{lem_gooddense} we may choose good angles
\begin{equation*}
\psi_1 \in \big(\psi(0) - 2C_1r - 8\eta,\; \psi(0) - 8\eta\big),
\qquad
\psi_2 \in \big(\psi(0) + 8\eta,\; \psi(0) + 2C_1r + 8\eta\big);
\end{equation*}
then pointwise $\Theta_{\psi_1} < \psi(0) < \Theta_{\psi_2}$ and
$\Theta_{\psi_1} < \Theta_{\psi_2}$. The band $R = \{(x, \psi) :
\Theta_{\psi_1}(x) < \psi < \Theta_{\psi_2}(x)\}$ has $\partial R =
T'_{\psi_1} \sqcup T'_{\psi_2}$. If the orbit ever left $\overline{R}$,
at the first boundary time it would lie on an invariant torus, forcing
the entire orbit---including its initial point---onto that torus, a
contradiction. So $\psi(t)$ lies in the arc $[\psi_1 - 4\eta, \psi_2 + 4\eta]$,
of length at most $4C_1 r + 24\eta \le (4C_1 + 24)(r + \eta)$.
\end{proof}

This is where two degrees of freedom enter: the tori have codimension
one in $\Sigma' \cong T^2 \times S^1$, and a \emph{pair} of graph-tori
separates it. The graph property here is free (KAM tori are $C^1$-small
graphs); the intrinsic statements are the Birkhoff-type graph
theorems for Lagrangian invariant tori
\cite{cite_BialyPolterovich1989LagrangianTori,
cite_Polterovich1991SecondBirkhoff, cite_Arnaud2010Birkhoff} and,
for merely continuous incompressible tori of geodesic flows on
$T^2$, Bialy's theorem \cite{cite_Bialy1989AubryMather}.

\subsection{Proof of Proposition \ref{prop_nearflat}}
\label{subsec_nearflat_proof}

\begin{proof}
\emph{Length bound.} Let $o$ be a closed $g'$-geodesic of class
$\beta$. As a Hamiltonian trajectory on $\Sigma'$ in the 
$\Psi$ coordinates this gives a trajectory $(x(t), p(t))$ of period $T =
\ell_{g'}(o)$, whose lift satisfies $\tilde x(T) - \tilde x(0) =
\beta$, so $\tfrac1T \int_0^T \dot x \, dt = \beta/T$
(recalling that $\beta $ also denotes the corresponding deck vector);
in particular $\beta \neq 0$, as the class is nontrivial. Now
\begin{equation*}
\dot x(t) = \partial_p H'(x,p) = p(t) + O(\varepsilon) =
(\cos \psi(t), \sin \psi(t)) + O(\varepsilon),
\end{equation*}
using $|p - (\cos \psi, \sin \psi)| \le C\varepsilon$. By Lemma
\ref{lem_confinement}, $\psi(t)$ lies in an arc $I$ of length $W$. Set
\begin{equation*}
e_* := \beta/|\beta|, \qquad \varphi_* := \arg e_*,
\end{equation*}
the direction of the total displacement. We claim $\varphi_*$ lies
within distance $4(W + C\varepsilon)$ of $I$. Indeed, fix any $\psi^0
\in I$; then $|\dot x(t) - (\cos \psi^0, \sin \psi^0)| \le W +
C\varepsilon$ for every $t$, hence the same bound holds for the
average $\beta/T$, and, as in the proof of Lemma
\ref{lem_confinement}, a vector within distance $s \le 1/4$ of a unit
vector makes angle at most $4s$ with it (smallness of $W +
C\varepsilon$ is arranged by the choice of $r, \eta, \varepsilon$
below). Consequently $|\psi(t) - \varphi_*| \le 5W + C\varepsilon$
for every $t$, and since $\psi \mapsto (\cos \psi, \sin \psi)$ is
$1$-Lipschitz,
\begin{equation*}
|\dot x(t) - e_*| \le C_3(W + \varepsilon) =: W'
\qquad \text{for all } t.
\end{equation*}
Since $|e_*| = 1$, this gives $\langle \dot x(t), e_* \rangle \ge 1 -
W'$ pointwise, so that
\begin{equation*}
T\, (1 - W') \;\le\; \int_0^T \langle \dot x(t), e_* \rangle \, dt
\;=\; \Big\langle \int_0^T \dot x(t) \, dt, \; e_* \Big\rangle
\;=\; \big\langle \beta, \beta/|\beta| \big\rangle
\;=\; |\beta| \;=\; L_\beta:
\end{equation*}
pairing with the direction of the displacement evaluates its length
exactly, with no inequality on the right. Whence $T \le L_\beta (1 -
W')^{-1}$. Choosing $r, \eta$ with
$C_3 C_2 (r + \eta) \le \delta/4$ and then $\varepsilon \le
\min\{\varepsilon_0(r, \eta),\, \delta/(4 C_3)\}$, we get
$W' \le \delta/2$, and since $(1 - \delta/2)^{-1} \le 1 + \delta$
for $\delta \in (0,1]$, the bound \eqref{eq_lengthbound} follows. The
confinement constrains all unit-speed geodesics simultaneously, so this
holds for every class $\beta$ with the same $\varepsilon$.

\emph{Tautness and the deformation.} By \eqref{eq_lengthbound},
$S(g', \beta) = S_{(1+\delta)L_\beta}(g', \beta)$, which is compact by
Lemma \ref{lem_SA}; so $g'$ is taut. For $g', g''$ in the
$\varepsilon$-ball, interpolate the squared norms, $F_t^2 = (1-t) F'^2
+ t F''^2$: strong convexity and $2$-homogeneity are preserved, and by
convexity of the ball each $g_t$ satisfies \eqref{eq_lengthbound}.
Hence for each $\beta$ the length spread satisfies
\begin{equation*}
\sup_{t \in [0,1]} \Big( \sup_{S(g_t, \beta)} \ell_{g_t} -
\inf_{S(g_t, \beta)} \ell_{g_t} \Big) \;\le\; (1+\delta)\, L_\beta
\;<\; \infty,
\end{equation*}
and Theorem \ref{thm:notgeodesible} shows $\{g_t\}$ is $\beta$-taut for
every $\beta$, i.e.\ a taut deformation. Taking $g'' = g_0$ gives the
last clause.
\end{proof}

\begin{remark}
\label{rem_nearflat_scope}
Dimension two is essential and enters only in Lemma
\ref{lem_confinement}: on $T^n$, $n \ge 3$, KAM tori have codimension
$n - 1 \ge 2$ in the energy level and separate nothing; the analogue of
Proposition \ref{prop_nearflat} is open there and adjacent to Arnold
diffusion. Nothing soft can replace the confinement: pointwise bounds
$|K| \le \varepsilon$ do not exclude long class $\beta$ geodesics---a
$\tau_\beta$-periodic plane curve of geodesic curvature $\le \varepsilon$
can execute U-turns at length cost $O(1/\varepsilon)$ apiece---so the
dynamical confinement does real work. Finally, the proof used only that
$H$ is fiberwise quadratic-convex and $x$-independent with
nonvanishing bordered determinant, so it applies verbatim to any constant-coefficient $g_0$ (as used in
the normalization above) and, more generally, to $g_0$ any flat
Minkowski (locally Minkowski Finsler) metric, whose indicatrix
replaces the round circle and supplies the radial-graph property, the
transversality in Lemma \ref{lem_gooddense}, and the twist.
\end{remark}

\section{Proof of Theorem \ref{thm_IntroTnH}}
\label{sec_Proof of Theorem thm_IntroTnH}
We prove a stronger result. 
\begin{theorem} \label{thm_IntroTn}
Let $g$ be the standard flat metric on $T ^{n}$, $n>1$, and
$\beta \in \pi _{1} (T ^{n})$ a nontrivial class.  % Suppose that $\beta \in \pi _{1} ^{inc} (X)$ is in the image
% of the inclusion $\pi _{1} ^{inc} (Z) \to \pi _{1} ^{inc} (Z
% \times Y)$.
For all $L$ sufficiently large, there exists an $\epsilon >0$
such that whenever $g'$ is a Finsler metric $C^2$ $\epsilon$-close
to $g$ and is $\beta$-regular, the following holds. 
Fix a prime $p$. Suppose that $o$ is a $g'$-geodesic
string such that:
\begin{enumerate}
\item The flat length satisfies $\ell _{g} (o) < L$. 
\item $o$ has class $\beta $.
\item $p | \mult (o) $.
\end{enumerate}
Then there is at least one other such $o$. 
\begin{itemize}
	\item Any  Finsler metric $g _{1}$ taut homotopic to the
	flat metric has the same property as $g$ above.
\item In particular, if $g _{1}$ is $\beta $-regular, is
$\beta $-taut homotopic to the flat metric and if there is a $g
_{1}$-geodesic string $o$ satisfying conditions (2), (3),
then there is another one.
\end{itemize}
\end{theorem}
\begin{proof} [Proof]
We work on the fixed contact manifold
\[
        C=S_g^*T^n,
\]
and let $h$ be a fixed auxiliary metric on $C$, for instance the Sasaki
metric associated to the flat metric $g$.  The nearby unit cotangent bundles
are identified with $C$ by the fixed radial convention of
Section~\ref{sec:Definition of F}; hence $R^{\lambda_{g'}}$ denotes the
corresponding pushed-forward Reeb vector field on $C$.

Let $L_\beta$ denote the flat length of a class $\beta$ geodesic, and let
$L>L_\beta$ be the length cutoff appearing in the statement.  Since $C$ is
compact and $R^{\lambda_{g'}}\to R^{\lambda_g}$ in $C^0$ for $g'$ sufficiently $C^2$-close to $g$, there is a preliminary $C^2$-neighborhood of $g$ and
constants
\[
        0<a<A<\infty
\]
such that, for every metric $g''$ in this neighborhood,
\[
        |R^{\lambda_{g''}}|_h\leq A,
        \qquad
        |\operatorname{pr}_*R^{\lambda_{g''}}|_g\geq a
\]
on $C$.  Set
\[
        A_L=\left(\frac{A}{a}+1\right)L.
\]
Define $\widehat U_{L,A_L}$ to be the set of smooth loops
$\gamma:S^1\to C$ such that
\[
        [\operatorname{pr}\circ\gamma]=\beta,
        \qquad
        \ell_g(\operatorname{pr}\circ\gamma)<L,
        \qquad
        \ell_h(\gamma)<A_L,
\]
and let
\[
        U_{L,A_L}=\widehat U_{L,A_L}/S^1
\]
be its image in the orbit-string quotient.  This is
a $h$-controlled neighborhood
of the compact lifted flat orbit block
\[
        N_0=\widetilde{\mathcal O}(g,\beta),
\]
in the sense of Definition \ref{def_hcontrolled}.
Indeed, the $h$-length upper bound is built into the definition of
$U_{L,A_L}$, and the nonzero projected class $\beta$ gives a positive lower
$h$-length bound on $\widehat U_{L,A_L}$. The suitable
compact set $K$ comes from Lemma
\ref{lemma_containedinCompact}.

Since $L>L_\beta$, we have $N_0\subset U_{L,A_L}$, and since
the flat metric has no other class $\beta$ geodesic strings,
$U_{L,A_L}$ is isolating for $N_0$.

Let $W$ be as in Lemma~\ref{lem_isolation}, and let $\epsilon$ be smaller
than both the corresponding Fuller constant for
\[
        (N_0,U_{L,A_L},W)
\]
and the size of the preliminary $C^2$-neighborhood above.
If $g'$ is sufficiently $C^2$ close to $g$, such that the corresponding Reeb vector fields are $C ^{0}$ $\epsilon $-close, Lemma~\ref{lem_continuation} gives
\begin{equation}\label{eq_vanishingF}
        F(g,\beta)
        =
        i(N_1,R^{\lambda_{g'}}),
\end{equation}
where
\[
        N_1=U_{L,A_L}\cap\mathcal O(R^{\lambda_{g'}}).
\]

We now identify the projection of $N_1$ with the length-filtered class $\beta$ geodesic
strings of $g'$.  Let $\gamma$ be a closed $R^{\lambda_{g'}}$-orbit,
parametrized by
\[
        \dot\gamma=T R^{\lambda_{g'}}(\gamma).
\]
Then
\[
        \ell_h(\gamma)
        \leq
        AT,
\]
while
\[
        \ell_g(\operatorname{pr}\circ\gamma)
        =
        T\int_0^1
        |\operatorname{pr}_*R^{\lambda_{g'}}(\gamma(t))|_g\,dt
        \geq
        aT.
\]
Thus
\[
        \ell_h(\gamma)
        \leq
        \frac{A}{a}\,
        \ell_g(\operatorname{pr}\circ\gamma).
\]
Consequently, for closed $R^{\lambda_{g'}}$-orbits,
\[
        \ell_g(\operatorname{pr}\circ\gamma)<L
        \quad\Longrightarrow\quad
        \ell_h(\gamma)<A_L.
\]
Since the definition of $U_{L,A_L}$ already includes the condition
$\ell_g(\operatorname{pr}\circ\gamma)<L$, 
the projection identifies $N_1$ bijectively with the set of class $\beta$
$g'$-geodesic strings satisfying $\ell_g(o)<L$.

Let $Z
\subset T ^{n}$ be a totally $g$-geodesic submanifold
diffeomorphic to $T ^{n-1}$, containing $\image \beta$. Let
$\alpha \in H ^{1} (T ^{n}, \mathbb{Z})$ be the Poincaré dual of $Z$ and
let $p: T ^{n} \to S ^{1}$ be the classifying map of
$\alpha$. Then clearly $p$ is a $\beta$-periodic fibration with
respect to $g$. 
By Theorem \ref{thm:EulerProduct}, $F (g,
\beta) =0$, since $\chi(S ^{1}) =0 $. And hence 
\begin{equation}\label{eq_IN=0}
  i (N _{1}, R
^{\lambda _{g'}}) = 0
\end{equation}
%.
By the previous paragraph,  we may rewrite
\eqref{eq_IN=0} as:
\begin{equation} \label{eq_vanishingmult}
\displaystyle \sum_{o \in \mathcal O_L(g', \beta )} \frac{(-1) ^{\operatorname
{morse}  (o)}} {\mult (o)} = 0,
\end{equation}
where $\mathcal O_L(g', \beta )$ denotes the set of
class $\beta $ $g'$-geodesic strings $o$ with $\ell _{g} (o)
< L$. Since $g'$ is $\beta$-regular and $N_1$ is compact,
$\mathcal O_L(g',\beta)$ is finite.  Enumerate its elements as
$o _{1}, \ldots, o _{N}$.  We then get:
\begin{equation*}
  \sum_{i} (-1) ^{\operatorname
{morse}  (o _{i})} \prod _{j \neq i} \mult (o _{j}) = 0.
\end{equation*}
%.
We show that if a prime $p$ divides the multiplicity
$\operatorname{mult}(o_k)$ for some $k \in \{1, \dots, N\}$,
then there exists at least one other index $m \neq k$ such
that $p$ divides $\operatorname{mult}(o_m)$. This will
complete the proof of the main statement. The two bulleted
addenda follow by the same argument, using that $F$ is a
taut deformation invariant, Theorem \ref{thm:invariantF},
and taking $L$ larger than the maximal flat length of a
class $\beta $ $g_1$-string (a finite quantity by tautness).

Let $M_i = \prod_{j \neq i} \operatorname{mult}(o_j)$ and
let $\sigma_i = (-1)^{\operatorname{morse}(o_i)}$. The
original equation is: $$\sum_{i=1}^{N} \sigma_i M_i = 0.$$
Suppose $p \mid \operatorname{mult}(o_k)$. Then
\begin{equation}\label{eq_pmod0}
  \sum_{i=1}^{N} \sigma_i M_i \equiv 0 \pmod p.
\end{equation}
%.
For any
index $i \neq k$, the product $M_i$ contains the factor
$\operatorname{mult}(o_k)$: $$M_i = \operatorname{mult}(o_k)
\cdot \prod_{j \neq i, k} \operatorname{mult}(o_j).$$ Since
$p \mid \operatorname{mult}(o_k)$, it follows from \eqref{eq_pmod0} that: 
$$\sigma_k M_k + 0 \equiv 0 \pmod p.$$  And so $$M_k
\equiv 0 \pmod p.$$ 

By definition, $M_k = \prod_{j \neq k}
\operatorname{mult}(o_j)$. Since $p$ is prime and $p \mid
\prod_{j \neq k} \operatorname{mult}(o_j)$, by Euclid's
Lemma, $$p \mid
\operatorname{mult}(o_j) \text{ for some } j \neq k.$$
And this completes the proof of the claim and the main
theorem.
\end{proof}
% \section{Proof of Theorem \ref{thm_introConjecture}}
% \label{sec_proofofThmthm_introConjecture}
% Let $p: T ^{n} \to S ^{1}$ be the $\beta$-periodic fibration
% with respect to $g_0$, as
% in the proof of Theorem \ref{thm_IntroTn} above.
% By Theorem \ref{thm:EulerProduct}, $F (g _{0},
% \beta) =0$. Hence, as we assumed Conjecture
% \ref{con_topological}, for any other $\beta $-regular $g$ on
% $T ^{n}$
% with finitely many $\beta$-class geodesic strings we have:
% \begin{equation*}
% \displaystyle \sum_{o \in
% \mathcal O(g,\beta)} \frac{(-1) ^{\operatorname
% {morse}  (o)}} {\mult (o)} = 0.
% \end{equation*}
% Then the conclusion of the Theorem follows by the same
% arithmetic argument as in the proof of Theorem
% \ref{thm_IntroTn} above.    
% \qed

\section{Reeb flows near the flat contact form}
\label{sec_ReebNearFlat}

We continue with the notation of Section \ref{sec_nearflat}: $T^2 =
\mathbb{R}^2/\mathbb{Z}^2$, $g_0 = G_0$ a flat metric, and $\nu, \psi,
W, W', C_3, C_4$ as in the proof of Proposition \ref{prop_nearflat}.
Set $C = S^*_{g_0}T^2$ with canonical contact form $\lambda_{g_0}$ and
contact structure $\xi = \ker \lambda_{g_0}$. Every contact form for
$\xi$ is $f\lambda_{g_0}$ with $f : C \to \mathbb{R}_{>0}$, and in the
fixed-$C$ convention of Section \ref{sec:Definition of F} one has
$\lambda_g = (1/F_g^*|_C)\, \lambda_{g_0}$ for any Riemann-Finsler
$g$. We use the flat trivialization $C \cong T^2 \times S^1 = T^3$,
$(x, \psi)$, under which free homotopy classes of loops in $C$ are
pairs $(\beta, w) \in \mathbb{Z}^2 \oplus \mathbb{Z}$, $\beta$ the
projected class and $w$ the fiber winding; the lift of a straight flat
class $\beta$ geodesic has class $(\beta, 0)$.

The first observation is that near $\lambda_{g_0}$, Reeb flows and
Finsler geodesic flows are the same objects.

\begin{lemma}[Convexity dictionary]
\label{lem_convexdictionary}
For every $k \geq 2$ there is $\varepsilon_c > 0$ such that if $\|f -
1\|_{C^k(C)} < \varepsilon_c$, then the fiberwise radial graph
\begin{equation*}
\Sigma_f \;=\; \{\, f(\xi)\, \xi \ : \ \xi \in C \,\} \;\subset\;
T^*T^2
\end{equation*}
is fiberwise strongly convex, and the $1$-homogeneous function $F^*$
with $F^*(t\xi) = t/f(\xi)$, $\xi \in C$, $t > 0$, is the dual norm of
a (in general non-reversible) Riemann-Finsler metric $g_f$ on $T^2$
with
\begin{equation*}
\lambda_{g_f} \;=\; f\, \lambda_{g_0},
\qquad \text{hence} \qquad
R^{f\lambda_{g_0}} \;=\; R^{\lambda_{g_f}} .
\end{equation*}
Moreover $g_f \to g_0$ in $C^{k-1}$ as $f \to 1$ in $C^k$.
\end{lemma}

\begin{proof}
$\Sigma_f$ is $C^2$-close to the $g_0$-unit cosphere bundle, hence
fiberwise strongly convex for $\varepsilon_c$ small; by
$1$-homogeneity $F^*$ is then a co-Finsler norm, and its Legendre dual
is the Finsler metric $g_f$, with the loss of one derivative coming
from the Legendre transform. Since $F_{g_f}^*|_C = 1/f$, the fixed-$C$
convention gives $\lambda_{g_f} = (1/F^*_{g_f}|_C)\lambda_{g_0} =
f\lambda_{g_0}$, and Reeb fields of equal contact forms are equal.
Reversibility of $g_f$ ($F(-\cdot) = F(\cdot)$) holds only when $f$ is
even in the fibers, but no statement of Section \ref{sec_nearflat}
uses reversibility: the arguments there require only fiberwise strong
convexity, star-shapedness, and the twist condition.
\end{proof}

\begin{corollary}
\label{cor_reebrigidity}
Let $k_0$ be as in Theorem \ref{thm_isoKAM}. For every $\delta \in
(0,1]$ there is $\varepsilon > 0$ such that if $\|f - 1\|_{C^{k_0+1}(C)}
< \varepsilon$, then, with $R = R^{f\lambda_{g_0}}$:
\begin{enumerate}
\item Every closed $R$-orbit has free homotopy class $(\beta, 0)$ for
some nontrivial $\beta$, and period at most $(1+\delta)\,|\beta|$. In
particular $R$ has no contractible closed orbits and no closed
orbits in the fiber classes.
\item For the linear path $f_t = (1-t)f + t$, the space of pairs
$(o, t)$, with $o$ a closed class $(\beta,0)$
$R^{f_t\lambda_{g_0}}$-orbit string, is compact for every $v$: the
family $\{R^{f_t\lambda_{g_0}}\}$ is a taut deformation from $R$ to
the flat geodesic Reeb field.
\item If, in addition, all closed class $(\beta, 0)$ $R$-orbits are
nondegenerate for a fixed nontrivial $\beta \in \pi_1(T^2)$, and one
orbit string $o$ satisfies $p \mid \mult(o)$ for a prime $p$, then
there is another such orbit string.
\end{enumerate}
\end{corollary}

\begin{proof}
By Lemma \ref{lem_convexdictionary}, $R^{f_t\lambda_{g_0}} =
R^{\lambda_{g_t}}$ for the Finsler metrics $g_t = g_{f_t}$, which lie
in the $C^{k_0}$-ball of Proposition \ref{prop_nearflat} (the ball in
$f$ is convex, and the dictionary loses one derivative). The Reeb flow
of $\lambda_{g_t}$ is, up to the radial identification, the
Hamiltonian flow of $H_t = \tfrac12 (F_t^*)^2$ on $\Sigma_t = \{H_t =
\tfrac12\}$, with periods equal to $g_t$-lengths.

(1) By Lemma \ref{lem_confinement}, along every orbit the momentum
angle $\psi(t)$ stays in an arc, so every closed orbit has fiber
winding $w = 0$. For a closed orbit of period $T$ and projected
projected class $\beta$, the estimate in the proof of Proposition
\ref{prop_nearflat} gives $|\beta| \geq T(1 - W')$; since $1 -
W' > 0$, the case $\beta = 0$ is impossible, and $T \leq |\beta| (1 -
W')^{-1} \leq (1+\delta)|\beta|$.

(2) By the dictionary, this space of pairs is the lifted cobordism of
the class $\beta$ orbit-string spaces of the Finsler family
$\{g_t\}$; by Section
\ref{subsec:lifted-geodesic-strings} it is homeomorphic to the
downstairs cobordism, which is compact since $\{g_t\}$ is a taut
deformation by Proposition \ref{prop_nearflat}.

(3) Nondegeneracy of the class $(\beta,0)$ closed $R$-orbits is, by
the dictionary, $\beta$-regularity of $g_f$ (as in the
Introduction), and $g_f$ is $\beta$-taut homotopic to the flat metric
by Proposition \ref{prop_nearflat}. Theorem \ref{thm_IntroTnH} applied
to $g_f$ gives a second geodesic string with $p$ dividing its
multiplicity, and the lift of Section
\ref{subsec:lifted-geodesic-strings} preserves classes and
multiplicities.
\end{proof}

The $L$-cutoff form of Theorem \ref{thm_IntroTn} needs neither the
dictionary nor any convexity: its proof uses the perturbed object only
through the $C^0$-distance of Reeb fields.

\begin{theorem}
\label{thm_ReebTn}
Let $g$ be the standard flat metric on $T^n$, $n > 1$, $C =
S^*_gT^n$, and $\beta \in \pi_1(T^n)$ a nontrivial class, of flat
length $L_\beta$. For every $L > L_\beta$ there exists $\epsilon > 0$
with the following property. Let $R$ be the Reeb vector field of any
contact form $f\lambda_g$ on $C$ with
\begin{equation*}
\sup_C\, |R - R^{\lambda_g}|_h \;<\; \epsilon,
\end{equation*}
and suppose that all closed $R$-orbit strings $o$ with projected
class $\beta$ and $\ell_g(\operatorname{pr} \circ o) < L$ are
nondegenerate. Fix a prime $p$. If one such orbit string satisfies $p
\mid \mult(o)$, then there is another one.
\end{theorem}

\begin{proof}
The proof of Theorem \ref{thm_IntroTn} applies verbatim. It uses the
perturbed field only through: the $C^0$ bounds $|R|_h \leq A$ and
$|\operatorname{pr}_* R|_g \geq a$ on $C$, which hold for $\epsilon$
small since they hold for $R^{\lambda_g}$; the isolation of the flat
orbit block and the Fuller constant of Lemmas \ref{lem_isolation} and
\ref{lem_continuation}, which are stated for arbitrary continuous
families of vector fields; the computation $F(g,\beta) = 0$, which
concerns the flat metric alone; and the identification of $N_1$ with
the orbit strings satisfying the two constraints, whose argument
($\ell_h \leq (A/a)\,\ell_g$ on closed orbits) is independent of any
metric structure on the perturbed side. Finally, by the Appendix, a
nondegenerate orbit string contributes $i(o)/\mult(o)$ to the Fuller
index with $i(o) \in \{\pm 1\}$ the fixed point index of its return
map, which is all the arithmetic of the proof requires. For geodesic
Reeb fields $i(\widetilde o) = (-1)^{\operatorname{morse}(o)}$
(Appendix \ref{appendix:Fuller}), recovering Theorem
\ref{thm_IntroTn}.
\end{proof}

\begin{remark}
\label{rem_gray}
Corollary \ref{cor_reebrigidity} and Theorem \ref{thm_ReebTn} also
cover Reeb fields of contact forms whose contact structure $\xi'$ is
merely $C^k$-close to $\xi$: by Gray stability
\cite[Theorem~2.2.2]{cite_Geiges2008} there is a $C^{k-1}$-small
diffeomorphism $\phi$, isotopic to the identity, with $\phi_*\xi' =
\xi$; conjugating $R$ by $\phi$ produces a Reeb field for a contact
form on $(C, \xi)$ close to $\lambda_{g_0}$, and $\phi$ preserves all
free homotopy classes.
\end{remark}

\begin{remark}
\label{rem_reebglobal}
The statements of this section are unconditional only near the flat
contact form. Away from it, a contact form $f\lambda_{g_0}$ need not
be fiberwise convex, Lemma \ref{lem_convexdictionary} fails, and
``taut homotopic to $R^{\lambda_{g_0}}$ through Reeb vector fields''
is a genuinely wider equivalence than the metric one. A Reeb analogue
of Theorem \ref{thm_IntroTnH} for this wider equivalence would
require a correspondingly stronger form of Conjecture
\ref{con_topological}.
\end{remark}

\appendix
\section{Fuller index and sky catastrophes}\label{appendix:Fuller}
In this appendix we recall the part of Fuller's periodic orbit index
used in the paper, and fix the notation for orbit strings and for
continuation under small perturbations of vector fields. Some of this
material overlaps with \cite{cite_SavelyevFuller}; we nevertheless
give a brief self-contained account, especially since the noncompact setting needed here is not
treated there, and the details serve to fix notation used throughout
the paper.

Let $X$ be a smooth manifold, and fix a complete auxiliary Riemannian metric
$h$ on $X$.  Let
\[
        \mathcal V(X)
\]
denote the space of smooth vector fields on $X$, with the $C^0$-topology on compact subsets.

Let
\[
        LX=C^\infty(S^1,X)
\]
denote the smooth free loop space.  For $V\in\mathcal V(X)$, define
\[
        S(V)
        =
        \left\{
        o\in LX
        \;\middle|\; o \text{ is nonconstant, }  
        \dot o(t)=T\,V(o(t))
        \text{ for some }T>0
        \right\}.
\]
We call it the period of $o$, and write
\[
        T=T(o),
\]
it is clearly uniquely determined.

The circle $S^1=\mathbb R/\mathbb Z$ acts on $S(V)$ by
reparametrization:
\[
        (s\cdot o)(t)=o(t+s).
\]
We define the space of orbit strings by
\[
        \mathcal O(V)=S(V)/S^1.
\]
If $\beta$ is a free homotopy class in $X$, we write
\[
        S(V,\beta)\subset S(V),
        \qquad
        \mathcal O(V,\beta)\subset \mathcal O(V)
\]
for the corresponding subspaces.  When no confusion is possible, we use the
same letter $o$ for a parametrized orbit and for its associated orbit
string.

Let
\[
        q:LX\to LX/S^1
\]
be the quotient map.  If
\[
        U\subset LX/S^1,
\]
we write
\[
        \widehat U=q^{-1}(U)\subset LX.
\]

\begin{definition}\label{def_hcontrolled}
We say that $U$ is \textbf{\emph{$h$-controlled}}  if there are constants
\[
        0<\ell_-<\ell_+<\infty
\]
and a compact set $K\subset X$ such that every loop $o\in\widehat U$
satisfies
\[
        \operatorname{image}(o)\subset K,
        \qquad
        \ell_-\leq \ell_h(o)\leq \ell_+.
\]

\end{definition}

The lower bound excludes collapse to constant loops; in the applications it is
usually supplied by a nontrivial projected free homotopy class.

For an orbit string $o\in \mathcal O(V)$, its multiplicity
\[
        m(o)\in\mathbb N
\]
is the unique positive integer such that $o$ is the $m(o)$-fold cover of a
primitive orbit string.  Equivalently, if $T(o)$ is the period of $o$ and
$T_{\mathrm{prim}}(o)$ is the period of the underlying primitive orbit
string, then
\[
        m(o)=\frac{T(o)}{T_{\mathrm{prim}}(o)}.
\]

\subsection{Fuller index}

Suppose
\[
        N\subset \mathcal O(V)
\]
is finite and consists of nondegenerate orbit strings.  Fuller associates to
$(N,V)$ the rational number
\[
        i(N,V)
        =
        \sum_{o\in N}\frac{i(o)}{m(o)}.
\]
Here $i(o)$ is the fixed point index of the time-$T(o)$ return map of the
flow of $V$, computed on a local hypersurface transverse to the image of
$o$.

By a compact open subset $N$ of an orbit-string space, we mean a subset which is compact and open in the relative topology.  
In this case Fuller again associates to this a rational
index $i(N,V)$. The latter can be defined perturbatively,
and using the formula as above. 

In the special case of a nondegenerate closed geodesic
orbit $o$, we have: 
\[
        i(\widetilde{o})=(-1)^{\operatorname{morse}(o)}.
\]
Consequently, if $N$ consists of finitely many
nondegenerate closed orbits of the geodesic Reeb flow $R
^{\lambda _{g}}$ on $C$, we get:
\[
        i(N,V)
        =
        \sum_{o\in N}
        \frac{(-1)^{\operatorname{morse}(o)}}{m(o)}.
\]

\subsection{Continuation}

Let
\[
        \{V_t\}_{t\in[0,1]}\subset \mathcal V(X)
\]
be a continuous family.  Define
\[
        S(\{V_t\})
        =
        \left\{
        (o,t)\in LX\times[0,1]
        \;\middle|\;
        o\in S(V_t)
        \right\},
\]
and
\begin{equation} \label{equation_Ohomotopy}
  \mathcal O(\{V_t\}) =
        S(\{V_t\})/S^1,
\end{equation}
where $S^1$ acts on the loop coordinate by
reparametrization as before.

Let
\[
        \mathcal N\subset \mathcal O(\{V_t\})
\]
be compact and open.  Suppose that its endpoint slices are
\[
        N_0=\mathcal N\cap \mathcal O(V_0),
        \qquad
        N_1=\mathcal N\cap \mathcal O(V_1).
\]
Then
\begin{equation}\label{eq_basicinvariance}
  i(N_0,V_0)=i(N_1,V_1).
\end{equation}
We refer to this as the \textbf{\emph{basic invariance}} of the Fuller index.

For applications, we need a local version of this statement under small
perturbations.  This is a variation of
~\cite[Lemma 4.1]{cite_FullerIndex}, with the added control needed for
noncompact $X$.

\begin{lemma}[Isolation lemma]\label{lem_isolation}
Let $h$ be a complete auxiliary Riemannian metric on $X$.  Let
$V_0\in\mathcal V(X)$, be a nonsingular vector field, and let
\[
        N_0\subset \mathcal O(V_0)
\]
be compact and open.  Let
\[
        U\subset LX/S^1
\]
be an $h$-controlled open neighborhood of $N_0$ such that
\[
        \overline U\cap\mathcal O(V_0)=N_0.
\]
Let $W$ be an open neighborhood of $N_0$ satisfying
\[
        \overline W\subset U.
\]
Then there exists $\epsilon>0$ with the following property.  
If $V$ is a smooth vector field satisfying

\[
        \sup_{x\in K}|V(x)-V_0(x)|_h<\epsilon,
\]
where $K$ is the compact set from the $h$-control of $U$, then
\[
        \mathcal O(V)\cap U\subset W.
\]
In particular,
\[
        \mathcal O(V)\cap U
\]
is compact and open in $\mathcal O(V)$.
\end{lemma}

\begin{proof}
Let $K$, $\ell_-$, and $\ell_+$ be the compact set and length bounds
from the $h$-control of $U$.  Since $V_0$ is nonsingular and $K$ is
compact, there are constants
\[
        0<B<A<\infty
\]
such that
\[
        B\leq |V_0(x)|_h\leq A
\]
for all $x\in K$.  We choose $\epsilon<B/2$.  Then every vector field
$V$ satisfying
\[
        \sup_K |V-V_0|_h<\epsilon
\]
obeys
\[
        |V(x)|_h\geq B/2,
        \qquad
        |V(x)|_h\leq A+\epsilon
\]
for all $x\in K$.

Suppose the conclusion failed for the chosen $W$.  Then there would be
vector fields
\[
        V_j\to V_0
\]
uniformly on $K$, and orbit strings
\[
        o_j\in \bigl(\mathcal O(V_j)\cap U\bigr)-W.
\]
Choose parametrized representatives, still denoted $o_j$, satisfying
\begin{equation}\label{eq_orbitequationT}
        \dot o_j(s)=T_j V_j(o_j(s)).
\end{equation}
 
Since $o_j\in\widehat U$, controlledness gives
\[
        \operatorname{image}(o_j)\subset K,
        \qquad
        \ell_-\leq\ell_h(o_j)\leq\ell_+.
\]
For $j$ sufficiently large,
\[
        B/2\leq |V_j|_h\leq A+\epsilon
\]
on $K$.  Hence
\[
        \frac{\ell_-}{A+\epsilon}\leq T_j\leq \frac{2\ell_+}{B}.
\]
The velocities are also uniformly bounded:
\[
        |\dot o_j(s)|_h
        =
        T_j|V_j(o_j(s))|_h
        \leq
        \frac{2\ell_+(A+\epsilon)}{B}.
\]
Thus the $o_j$ are equicontinuous and have images in the fixed compact set
$K$.  By the Arzelà--Ascoli theorem, after passing to a subsequence,
\[
        o_j\to o_\infty
\]
uniformly.  After passing to a further subsequence, $T_j\to T_\infty$, with $T_\infty>0$.  

Hence $o _{j}$ is a Cauchy sequence in the
$C ^{1}$ space of maps, by the defining equation \eqref{eq_orbitequationT}.  Then $C ^{1}$ completeness gives that
\[
        \dot o_\infty=T_\infty V_0(o_\infty).
\]
Thus $o_\infty$ is a closed $V_0$-orbit string.  Moreover
\[
        o_\infty\in \overline U-W.
\]
This contradicts
\[
        \overline U\cap\mathcal O(V_0)=N_0\subset W.
\]
Therefore
\[
        \mathcal O(V)\cap U\subset W
\]
for all $V$ sufficiently close to $V_0$ on $K$.

It remains to prove compactness.  Let $o_j$ be any sequence in
$\mathcal O(V)\cap U$.  The same estimates as above give uniformly bounded
periods and velocities, with all images contained in $K$.  Arzel\`a--Ascoli
therefore gives a convergent subsequence.  Its limit is again a closed
$V$-orbit.  Since the inclusion just proved gives
\[
        \mathcal O(V)\cap U
        =
        \mathcal O(V)\cap\overline W
\]
and $\overline W\subset U$, the limit lies in
$\mathcal O(V)\cap U$.  Thus $\mathcal O(V)\cap U$ is compact.  It is open
because $U$ is open.
\end{proof}

We call such an $\epsilon$ a \textbf{\emph{Fuller constant}} for the data
\[
        (V_0,N_0,U,W).
\]
It also depends on the auxiliary control data $K,\ell_-,\ell_+$, which will
always be understood from the chosen $U$.

\begin{lemma}[Fuller continuation]\label{lem_continuation}
Let $V_0$, $N_0$, $U$, $W$, and $\epsilon$ be as in
Lemma~\ref{lem_isolation}.  Let
\[
        \{V_t\}_{t\in[0,1]}, \, V_{t=0}=V_0
\]
be a continuous family such that for every $t\in[0,1]$,
\[
        \sup_{x\in K}|V_t(x)-V_0(x)|_h<\epsilon,
\]
where $K$ is the compact set from the $h$-control of $U$.  Define
\[
        \mathcal N
        =
        (U\times[0,1])\cap \mathcal O(\{V_t\}).
\]
Then $\mathcal N$ is compact and open in the orbit cobordism
$\mathcal O(\{V_t\})$.  Its endpoint slices are
\[
        N_0=U\cap\mathcal O(V_0),
        \qquad
        N_1=U\cap\mathcal O(V_1).
\]
Consequently,
\[
        i(N_0,V_0)=i(N_1,V_1).
\]
We call $N_1$ the Fuller continuation of $N_0$.
\end{lemma}

\begin{proof}
Since $U$ is open,
\[
        \mathcal N
        =
        (U\times[0,1])\cap \mathcal O(\{V_t\})
\]
is open in $\mathcal O(\{V_t\})$.

By Lemma~\ref{lem_isolation}, applied to each $V_t$, we have
\[
        \mathcal O(V_t)\cap U\subset W
\]
for every $t\in[0,1]$.  Since
\[
        \overline W\subset U,
\]
this implies
\[
        \mathcal N
        =
        (\overline W\times[0,1])\cap \mathcal O(\{V_t\}).
\]
The compactness then works as in 
Lemma~\ref{lem_isolation}. 
% Indeed, for a sequence $(o_j,t_j)\in\mathcal N$, compactness of
% $[0,1]$ gives $t_j\to t_\infty$ after passing to a subsequence, and the
% same period and velocity bounds as in Lemma~\ref{lem_isolation} give
% $o_j\to o_\infty$ in $C ^{0}$ norm. Following the argument
% in Lemma \ref{lem_isolation} we upgrade this to $C ^{1}$
% convergence, which also means that $o _{\infty} \in
% O(\{V_t\})$ and hence $\mathcal{N}$ as the latter is closed.   
The endpoint slices are
\[
        \mathcal N\cap\mathcal O(V_0)
        =
        U\cap\mathcal O(V_0)
        =
        N_0
\]
and
\[
        \mathcal N\cap\mathcal O(V_1)
        =
        U\cap\mathcal O(V_1)
        =
        N_1.
\]
Fuller's basic invariance applied to the compact open cobordism
$\mathcal N$ gives
\[
        i(N_0,V_0)=i(N_1,V_1).
\]
\end{proof}
\begin{example}[Geodesic flows and the Sasaki metric]\label{exm_sasaki}
Let $(Y,g_Y)$ be a complete Riemannian manifold, and let
\[
        C _{g}=S^*Y
\]
be its unit cotangent bundle with Liouville contact form $\lambda_{g_Y}$.
The Reeb vector field $R^{\lambda_{g_Y}}$ is the geodesic Reeb field.

Let $g_S$ be the Sasaki metric on $S^*Y$.  It is defined using the
Levi-Civita connection of $g_Y$.  Thus
\[
        TC=T^{\mathrm{vert}}C\oplus T^{\mathrm{hor}}C
\]
is an orthogonal splitting, where
\[
        T^{\mathrm{vert}}C=\ker(\operatorname{pr}_*)
\]
and $T^{\mathrm{hor}}C$ is the Levi-Civita horizontal subbundle.  If
$\xi,\eta\in T(S^*Y)$ are horizontal, then
\[
        g_S(\xi,\eta)=g_Y(\operatorname{pr}_*\xi,\operatorname{pr}_*\eta),
\]
and the vertical part is measured using the metric induced by $g_Y$ on the
cotangent fibers.

The basic property needed in the paper is the following.  If $o$ is a
nonzero constant-speed $g_Y$-geodesic and $\widetilde o$ is its unit cotangent
lift, then
\begin{equation}\label{eq_sasaki}
        \ell_{g_S}(\widetilde o)=\ell_{g_Y}(o).
\end{equation}
Indeed, the covector component of $\widetilde o$ is parallel along $o$, so the
vertical part of $\dot{\widetilde o}$ vanishes and the Sasaki speed equals the
base speed.

If $\beta\in\pi_1^{\mathrm{inc}}(Y)$ and $U\subset S(g_Y,\beta)$ is
$g_Y$-length bounded, then the set of canonical lifts
\[
        \widetilde U\subset S(R^{\lambda_{g_Y}})
\]
is $g_S$-length bounded by \eqref{eq_sasaki}.
\end{example}
\section{\texorpdfstring{A local fibered calculation of $F$}{A local fibered calculation of F}} \label{sec_A local fibered calculation}
We prove in this section the technical lemma computing the
local index, used in Theorem \ref{thm:EulerProduct}.

Let
\[
        P:S^*X\to\mathbb R
\]
be the verticality function as in the proof of
Theorem~\ref{thm:EulerProduct},
\[
        P(\xi)=|\pi^{\operatorname{vert}}v_\xi|_g^2.
\]
Here $v_\xi$ is the $g$-dual unit vector to $\xi$, and
\[
        \pi^{\operatorname{vert}}:TX\to T^{\operatorname{vert}}X
\]
is the $g$-orthogonal projection.  The notation
$\beta$ and $\beta_Z$ is likewise as in
Theorem~\ref{thm:EulerProduct}.

\begin{lemma}[Local product calculation]\label{lem_fuller-local-product-geodesic}
Let
\[
        Z\hookrightarrow X\xrightarrow{p}Y
\]
be a parallel Riemannian submersion.  Let
\[
        \widetilde f=f\circ p\circ\operatorname{pr}
\]
for a Morse function $f:Y\to\mathbb R$, and define the vector field on
$S^*X$
\[
        V_\varepsilon
        =
        R^{\lambda_g}
        -
        \varepsilon\nabla_{g_S}(P+\widetilde f).
\]
Fix
\[
        y\in\operatorname{Crit}(f),
\]
and let
\[
        N_{\varepsilon,y}
\]
be the isolated set of closed $V_\varepsilon$-orbit strings lying over the
fiber
\[
        Z_y=p^{-1}(y),
\]
as in the proof of Theorem~\ref{thm:EulerProduct}.  Then
\[
        i(N_{\varepsilon,y},V_\varepsilon)
        =
        (-1)^{\dim Y+\operatorname{morse}(y)}
        \sum_{\alpha\in S_y}F(g_y,\alpha),
\]
where $S_y$ is the $\pi_1(Y,y)$-orbit of the corresponding fiber
representative in $\pi_1^{\mathrm{inc}}(Z_y)$.  In particular, if the fibration is holonomy-finite in the relevant class, the sum is finite.
\end{lemma}

\begin{proof}
Let
\[
        T^{\operatorname{vert}}X=\ker(dp),
        \qquad
        H=(T^{\operatorname{vert}}X)^{\perp_g}.
\]
Since the submersion is parallel, the two distributions
\[
        T^{\operatorname{vert}}X
        \quad\text{and}\quad
        H
\]
are preserved by the Levi-Civita connection.  
Since $T^{\operatorname{vert}}X$ and $H=(T^{\operatorname{vert}}X)^{\perp_g}$
are parallel complementary distributions, the local de Rham decomposition
theorem gives a product
decomposition
\[
        p^{-1}(D_y)\cong Z_y\times D_y,
        \qquad
        g=g_y\oplus g_Y| _{D _{y}},
\]
for a normal ball neighborhood $D _{y} \subset Y$ centered
at $y$. 
See Kobayashi--Nomizu~\cite[Chapter
IV]{cite_KobayashiNomizu1963}.

For this local calculation we suppress the fixed-$C$ identification and
write $S^*X$ for the $g$-unit cotangent bundle.
Let
\[
        C^{\operatorname{vert}}=\{P=1\}\subset S^*X
\]
be the vertical locus.  Over $Z_y$, this locus is naturally identified with
\[
        S^*Z_y\subset S^*X.
\]
Near
\[
        S^*Z_y\subset C^{\operatorname{vert}},
\]
the product decomposition gives the normal-bundle model
\begin{equation}\label{eq_localmodel}
        S^*Z_y\times D_y\times B_y^*,
\end{equation}
where
\[
        B_y^*\subset T_y^*Y
\]
is a small ball in the horizontal cotangent direction.  In this model, the
vertical locus is
\[
        S^*Z_y\times D_y\times\{0\}.
\]

A covector near the vertical locus decomposes as
\[
        \xi=\xi_{\operatorname{vert}}+\eta,
        \qquad
        \eta\in B_y^*.
\]
The unit condition is
\[
        |\xi_{\operatorname{vert}}|_{g_y^*}^2
        +
        |\eta|_{g_Y^*}^2
        =
        1.
\]
Therefore
\[
        P(\xi)
        =
        |\xi_{\operatorname{vert}}|_{g_y^*}^2
        =
        1-|\eta|_{g_Y^*}^2.
\]
Thus $P$ has a Morse--Bott maximum along $C^{\operatorname{vert}}$.
% Consequently the normal
% fixed point index contribution is
% \[
%         (-1)^{\dim Y}.
% \]

In the local model \eqref{eq_localmodel}, 
\[
        N_{\varepsilon,y}
        =
        N_y^{\operatorname{fib}}\times\{y\}\times\{0\},
\]
where $N_y^{\operatorname{fib}}$ is the relevant closed orbit set of the
fiber geodesic Reeb flow on $S^*Z_y$.

We now describe the local return map at an element of $N_{\varepsilon,y}$.
Along the vertical locus over $Z_y$, the fiber component of
$V_\varepsilon$ is the geodesic Reeb field
        $R^{\lambda_{g_y}}$
on $S^*Z_y$.  The base component is
\[
        -\varepsilon\nabla_{g_Y}f
\]
on $D_y$.  The normal covector component is the outward radial component
coming from
\[
        -\varepsilon\nabla_{g_S}P
\]
on $B_y^*$.

Let
\[
        o\in N_y^{\operatorname{fib}},
\]
and choose a product transverse section
\[
        \Sigma
        =
        \Sigma^{\operatorname{fib}}\times D_y\times B_y^*,
\]
where
\[
        \Sigma^{\operatorname{fib}}\subset S^*Z_y
\]
is transverse to the fiber Reeb flow near $o$.  With respect to this product
section, the local return map of $V_\varepsilon$ is the product
\[
        \Phi_{\operatorname{fib}}\times\psi_y\times\kappa.
\]
Here:
\begin{itemize}
        \item $\Phi_{\operatorname{fib}}$ is the local return map of the
        fiber geodesic flow generated by $R^{\lambda_{g_y}}$;
        \item $\psi_y$ is the small-time map of
        \[
                -\varepsilon\nabla_{g_Y}f
        \]
        near $y$;
        \item $\kappa$ is the normal return map induced by
        \[
                -\varepsilon\nabla_{g_S}P
        \]
        near $0\in B_y^*$.
\end{itemize}

The fixed point index is multiplicative under products,
Dold~\cite[Chapter VII, Section 5]{cite_Dold1972}.  Hence
\[
        i_{\operatorname{fp}}(o;V_\varepsilon)
        =
        i_{\operatorname{fp}}(o;R^{\lambda_{g_y}})
        \cdot
        i_{\operatorname{fp}}(\psi_y,y)
        \cdot
        i_{\operatorname{fp}}(\kappa,0).
\]
Since $\psi_y$ is the positive small-time map of
$-\varepsilon\nabla_{g_Y}f$, the map
$\operatorname{id}-\psi_y$ has the same local degree as
$\nabla_{g_Y}f$.  Thus, because $y$ is a Morse critical point,
\[
        i_{\operatorname{fp}}(\psi_y,y)
        =
        (-1)^{\operatorname{morse}(y)}.
\]
Since the normal component of
\[
        -\varepsilon\nabla_{g_S}P
\]
has outward radial linearization on a vector space of dimension $\dim Y$,
the map $\operatorname{id}-\kappa$ has local degree $(-1)^{\dim Y}$.
Hence
\[
        i_{\operatorname{fp}}(\kappa,0)
        =
        (-1)^{\dim Y}.
\]
Therefore
\[
        i_{\operatorname{fp}}(o;V_\varepsilon)
        =
        (-1)^{\dim Y+\operatorname{morse}(y)}
        i_{\operatorname{fp}}(o;R^{\lambda_{g_y}}).
\]

The multiplicity of $o$ is unchanged by taking the product with the isolated
base and normal fixed points.  Thus
\[
        \frac{i_{\operatorname{fp}}(o;V_\varepsilon)}{m(o)}
        =
        (-1)^{\dim Y+\operatorname{morse}(y)}
        \frac{i_{\operatorname{fp}}(o;R^{\lambda_{g_y}})}{m(o)}.
\]
Summing over all fiber orbit strings in the local block gives
\[
        i(N_{\varepsilon,y},V_\varepsilon)
        =
        (-1)^{\dim Y+\operatorname{morse}(y)}
        \sum_{\alpha\in S_y}F(g_y,\alpha).
\]
This proves the claim.
\end{proof}

\section{Sky catastrophes} \label{appendix_bluesky}
We now recall the topological formulation of sky
catastrophes as studied in ~\cite{cite_SavelyevFuller}. This 
is the formulation necessary for Section 
\ref{sec_digression}.
\begin{definition} \label{def:bluesky}
Let
\[
        \{V_t\}_{t\in[0,1]}\subset\mathcal V(X)
\]
be a continuous family.  We say that the family has a sky catastrophe in class
$\beta$ if there is an endpoint orbit string
\[
        y\in \mathcal O(V_0,\beta)\cup \mathcal O(V_1,\beta)
        \subset \mathcal O(\{V_t\},\beta)
\]
such that no compact open subset of $\mathcal O(\{V_t\},\beta)$ contains $y$.

Equivalently, the endpoint orbit string $y$ cannot be included in any
compact continuation block through the family.
\end{definition}

For suitably regular families, this agrees with the preliminary
Definition \ref{def_preliminarysky}.
\section*{Acknowledgements} 
I am grateful to the IAS for resources provided during my
stay as a member, where some of this paper was written.
Thanks also to Mohammed Abouzaid, Semon Rezchikov, Viktor
Ginzburg and Yasha Eliashberg for interesting
related discussions, and to Leonid Polterovich for comments on the history of the graph property of invariant tori.
\bibliographystyle{siam} 
\bibliography{link.bib} 
% \bibliography{/root/texmf/bibtex/bib/link}  
% \bibliography{/home/yashasavelyev/texmf/bibtex/bib/link} 
% \bibliography{/home/yasha/texmf/bibtex/bib/link} 
\end{document}